\documentclass{amsart}
\usepackage[T1]{fontenc}
\usepackage{amsfonts}
\usepackage{mathrsfs}
\usepackage{mathtools}
\usepackage{amssymb}
\usepackage[colorlinks=true,linkcolor=blue,allcolors=blue]{hyperref}
\usepackage[noabbrev,capitalize,nameinlink]{cleveref}
\usepackage{amsmath,blkarray,booktabs,bigstrut}
\usepackage{graphicx}
\usepackage{dutchcal}
\usepackage{algorithm}
\usepackage{algpseudocode}
\usepackage{tikz}
\usepackage{tikz-cd}
\usepackage{graphicx}
\usepackage[normalem]{ulem}
\numberwithin{equation}{section}

\newcommand\R{{\mathbb R}}

\renewcommand\min{{\operatorname{min}}}

\DeclareMathOperator{\supp}{supp}

\DeclareMathOperator{\kernel}{Ker}
\DeclareMathOperator{\range}{Ran}
\DeclareMathOperator{\Dim}{Dim}
\DeclareMathOperator{\trace}{trace}
\DeclareMathOperator{\spanning}{span}

\newcommand\bR{{\mathbb R}}
\newcommand{\bP}{{\mathbf P}}
\newcommand{\bQ}{{\mathbf Q}}
\newcommand{\bpi}{{\boldsymbol{\pi}}}
\newcommand{\bS}{{\mathbf S}}
\newcommand{\bw}{{\mathbf w}}
\newcommand{\bv}{{\mathbf v}}
\newcommand{\bof}{{\mathbf f}}

\newcommand{\bdel}{{\boldsymbol{\delta}}}
\newcommand{\bphi}{{\boldsymbol{\phi}}}
\newcommand{\cH}{{\mathcal H}}
\newcommand{\cP}{{\mathcal P}}
\newcommand{\rar}{\rightarrow}
\newcommand{\ip}[2]{\left\langle#1,#2\right\rangle}

\newcommand{\abs}[1]{\left\vert#1\right\vert}
\newcommand{\absip}[2]{\left\vert\langle#1,#2\rangle\right\vert}

\theoremstyle{plain}
  \newtheorem{theorem}{Theorem}[section]
  \newtheorem{proposition}[theorem]{Proposition}
  \newtheorem{lemma}[theorem]{Lemma}
  \newtheorem{corollary}[theorem]{Corollary}

  \theoremstyle{remark}

\theoremstyle{definition}
  \newtheorem{definition}[theorem]{Definition}
  
  \newtheorem{example}[theorem]{Example}

\begin{document}
\include{psfig}
\title[A Land of Oblique Duality for Frames]{A Land of Oblique Duality for Frames and Probabilistic Frames}


\author{Dongwei Chen}
\address{Department of Mathematics, Colorado State University, Fort Collins, CO, USA, 80523}
\email{dongwei.chen@colostate.edu}

\author{Emily J. King}
\address{Department of Mathematics, Colorado State University, Fort Collins, CO, USA, 80523}
\email{emily.king@colostate.edu}

\author{Clayton Shonkwiler}
\address{Department of Mathematics, Colorado State University, Fort Collins, CO, USA, 80523}
\email{clayton.shonkwiler@colostate.edu}
 
\begin{abstract}
Functions or distributions used to sample and to reconstruct signals often occur in different domains, like the Dirac delta and a band-limited bump function in classical sampling. Oblique dual frames generalize this phenomenon.
In this paper, we provide new tools to study oblique dual frames and introduce a probabilistic variant of oblique dual frames. 
We first present the oblique dual frame potential and show that it is 
minimized precisely when the oblique dual coincides with the canonical 
oblique dual. We then define oblique dual probabilistic frames and 
oblique approximately dual probabilistic frames. In particular, we prove 
that for a given oblique dual probabilistic frame, the associated oblique 
dual probabilistic frame potential is minimized if and only if the frame 
is tight and the oblique dual is canonical. Moreover, the tightness 
assumption can be removed when the minimization is restricted to oblique 
dual probabilistic frames of pushforward type. Finally, we investigate 
perturbations of oblique dual probabilistic frames and show that if a 
probability measure is sufficiently close to an oblique dual 
probabilistic frame pair in the $2$-Wasserstein topology, then it forms an 
oblique approximately dual probabilistic frame.
\end{abstract}

\maketitle

\section{Introduction}
Suppose $\mathcal{H}$ is a separable Hilbert space and $W \subset \mathcal{H}$ is a closed subspace. An at-most countable sequence $\{\mathbf{w}_i\}_{i \in I} \subset W$ is called a \emph{frame} for $W$ if there exist constants $0<A \leq B $ such that for any $\mathbf{f} \in W$, 
\begin{equation}\label{eq:frame}
    A \|\mathbf{f}\|^2 \leq \sum_{i\in I} |\langle \mathbf{w}_i, \mathbf{f} \rangle|^2 \leq B \|\mathbf{f}\|^2.
\end{equation}
A frame $\{\mathbf{w}_i\}_{i \in I}$ is called a \emph{tight frame} for $W$ if we may choose $A=B$ and \emph{Parseval} if $A=B=1$. If the upper (but possibly not the lower) bound in~\eqref{eq:frame} holds, $\{\mathbf{w}_i\}_{i \in I}$ is called a \emph{Bessel sequence} for $W$.

Given any frame $\{\mathbf{w}_i\}_{i\in I}$ for $W$, one can use \emph{dual frames} to reconstruct vectors from $W$: There exists another frame $\{\mathbf{v}_i\}_{i\in I}$ in $W$ such that for any $\mathbf{f} \in W$, 
\begin{equation}\label{eq:dualframe}
    \mathbf{f} = \sum_{i\in I} \langle \mathbf{f}, \mathbf{w}_i \rangle \mathbf{v}_i = \overbrace{\sum_{i\in I} \underbrace{\langle \mathbf{f}, \mathbf{v}_i \rangle}_{\text{sampling}} \mathbf{w}_i}^{\text{reconstruction}}.
\end{equation}
In particular, one can use the \emph{canonical dual frame} $ \{\mathbf{S}^\dagger\mathbf{w}_i\}_{i\in I}$ where  $\mathbf{S}^\dagger$ is the Moore–Penrose inverse of the frame operator $\mathbf{S}: \mathcal{H} \rightarrow \mathcal{H}$ for $\{\mathbf{w}_i\}_{i\in I}$ given by $ \mathbf{S} \mathbf{f} = \sum_{i\in I} \langle \mathbf{f}, \mathbf{w}_i \rangle \mathbf{w}_i$, for any $\mathbf{f} \in \mathcal{H}$. 

Frames were first introduced by Duffin and Schaeffer to analyze perturbations
of Fourier series ~\cite{duffin1952class} and have been widely used in signal processing ~\cite{daubechies1992ten, grochenig2001foundations,  strohmer2003grassmannian, kutyniok2012multiscale, naidu2020construction}, where signals are modeled as vectors in a Hilbert space $\mathcal{H}$, and signal measurement or sampling is defined as taking the inner product with frame vectors, and signal reconstruction is understood as the weighted sum of samples with dual (or tight) frame vectors \cite{eldar2005general}. 
Most of the current frame literature also calls signal sampling and reconstruction signal analysis and synthesis, respectively. When $W$ is finite-dimensional, the number of vectors in a frame for $W$ is necessarily finite, which we call a \emph{finite frame}. 
Interested readers can refer to \cite{christensen2016introduction, waldron2018introduction} for more details on frames and finite frames.

As seen in \cref{eq:dualframe}, the sampling and reconstruction procedures of dual frames occur in the same subspace $W$, because $\{\mathbf{w}_i\}_{i \in I}$ and $\{\mathbf{v}_i\}_{i \in I}$ are both contained in $W$. However, in practice, the sampling and reconstruction can be performed across distinct subspaces.

\begin{example}[{\cite[Example 2.1]{li1998theory}} and {\cite[Example 1]{li2004pseudo}}]\label{example:paley}
Define the signal space $PW_{1/4}$ to be the space of all band-limited signals of bandwidth within $\left[-1/4, 1/4\right)$. 
By the Whittaker–Nyquist–Shannon–Someya sampling theorem, there exists a function 
$\bphi \in L^{2}(\mathbb{R})$ such that its Fourier transform $\hat{\bphi}$ satisfies
\[
\hat{\bphi}(\gamma) = 
\begin{cases}
1, & -\tfrac{1}{4} \leq \gamma < \tfrac{1}{4}, \\[6pt]
\text{decaying continuously to zero}, & \tfrac{1}{4} \leq |\gamma| < \tfrac{1}{2}, \\[6pt]
0, & |\gamma| \geq \tfrac{1}{2},
\end{cases}
\]
and for $T = 1$ (which satisfies the Nyquist rate $2T \cdot 1/4 < 1$), we have
\[
\bof(t) = \sum_{n=-\infty}^\infty \bof(n)\,\bphi(t-n) = \overbrace{\sum_{n=-\infty}^\infty \underbrace{ \langle \bof, \bdel_n \rangle}_{\text{sampling}} \bphi(t-n)}^{\text{reconstruction}}, \quad \textrm{for all} \quad \bof \in PW_{1/4}.
\]
Note that the sampling space is $\overline{\operatorname{span}}\{\bdel_n:n\in\mathbb{Z}\}$, which is a subspace of the tempered distribution space. However, the reconstruction space is $\overline{\operatorname{span}}\{\bphi(\cdot-n):n\in\mathbb{Z}\}\subset L^2(\mathbb{R})$ rather than $PW_{1/4}$, because   $\bphi \notin PW_{1/4}$ and $\{\bphi(\cdot-n)\}_{n \in \mathbb{Z}}$ is not a frame for $PW_{1/4}$. Thus, the sampling vectors belong to a distributional subspace, while the reconstruction vectors lie in $L^{2}(\mathbb{R})$. 
\end{example}
Further examples of sampling and reconstructing in different spaces include reconstructing after irregular sampling~\cite{aldroudi2002nonuniform,feichtinger1992iterative}, correcting for sensor issues~\cite{unser2002general}, 
using a frame-like construction with certain ``niceness'' (e.g., smoothness) of the frame-like vectors not otherwise possible~\cite{li1998theory,christensen2004oblique}, and decomposing Besov spaces~\cite{frazier1985decomp}. These and other considerations led to the development of \emph{oblique duals}~\cite{eldar2001nonredundant,eldar2003sampling,christensen2004oblique,eldar2005general, eldar2006characterization} and the related \emph{pseudoframes for subspaces (PFFS)}~\cite{li1998theory,li1998pseudo,li2004pseudo}. Subsequent developments of oblique dual frames appear in, e.g., \cite{heineken2018oblique, diaz2023approximate,li2024making,li2025approximate}.

For a different perspective, consider a frame $\{\mathbf{w}_i\}_{i \in I}$ for $W$ with dual frame $\{\mathbf{v}_i\}_{i \in I}$ in $W$.  Then, it follows from \cref{eq:dualframe} that the map $\bP_W: \cH \rar \cH$ defined by 
\begin{equation}\label{eq:orthproj}
 \bP_W \mathbf{f} = \sum_{i\in I} \langle \mathbf{f}, \mathbf{w}_i \rangle \mathbf{v}_i
\end{equation}
satisfies $\bP_W^2 = \bP_W$ with $\range(\bP_W) = W$ and $W^\perp \subseteq \kernel(\bP_W)$.  In particular, $\bP_W$ is the orthogonal projection onto $W$.  A natural question to ask is whether one can define oblique projections with a formula similar to \cref{eq:orthproj}.

The mathematical setting for oblique dual frames is as follows. Let $\cH$ be a separable Hilbert space and $I$ an at-most countable index set. Suppose $W$ and $V$ are closed subspaces of $\mathcal{H}$ such that $\mathcal{H} = W \oplus V^\perp$ where $V^\perp$ is the orthogonal complement of $V$. Let $\{\mathbf{w}_i\}_{i \in I} \subset W$ and $\{\mathbf{v}_i\}_{i \in I} \subset V$ be Bessel sequences for $\mathcal{H}$ and further assume that  $\{\mathbf{w}_i\}_{i \in I}$ and $\{\mathbf{v}_i\}_{i \in I}$ are frames for $W$ and $V$, respectively. 
Then $\{\mathbf{v}_i\}_{i \in I}$ is an \emph{oblique dual frame} of $\{\mathbf{w}_i\}_{i \in I}$ on $V$ if 
\begin{equation*}
   \bpi_{WV^\perp} \mathbf{f} = \sum_{i \in I} \langle \mathbf{f}, \mathbf{v}_i \rangle \mathbf{w}_i, \quad \text{for all} \quad \mathbf{f} \in \mathcal{H},
\end{equation*}
where the map 
\[
\bpi_{WV^\perp}: \cH \rar \cH, \quad \bof \mapsto \sum_{i \in I} \langle \mathbf{f}, \mathbf{v}_i \rangle \mathbf{w}_i
\]
is the oblique projection of $\cH$ onto $W$ along $V^\perp$.
Furthermore, $\{\mathbf{v}_i\}_{i \in I}$ are called the \emph{sampling} (or \emph{analysis}) \emph{vectors} and  $V$ the \emph{sampling space}. Similarly, $\{\mathbf{w}_i\}_{i \in I}$ are called the \emph{reconstruction} (or \emph{synthesis}) \emph{vectors} and $W$ the \emph{reconstruction space}. Note that when $V=W$, the definition of oblique dual frame reduces to the standard dual frame setting. When $\mathcal{H}$ is finite-dimensional, the frames are finite, and the index set $I$ is finite.

Our contributions in this paper focus on the study of oblique dual frame potentials and their probabilistic counterparts. The paper is organized as follows.   In \cref{section:preliminaries}, we give some background and preliminary results on oblique dual frames, probabilistic frames, and optimal transport, which serves as a key tool in the study of probabilistic frames. In \cref{section:dualptential}, we study the oblique dual frame potential where $\mathcal{H} = \mathbb{C}^n=W \oplus V^\perp$. In this setting, suppose $\{\mathbf{w}_i\}_{i=1}^N \subset W$ is a frame for $W$ with the frame operator $\mathbf{S} = \sum_{i=1}^N \mathbf{w}_i\mathbf{w}_i^*$, and $\{\mathbf{v}_i\}_{i=1}^N \subset V$ is an oblique dual frame of $\{\mathbf{w}_i\}_{i=1}^N$ on $V$. In \cref{lemma:oblique_dual_potential}, we show that 
    \begin{equation*}
       \sum_{i=1}^N \sum_{j=1}^N |\langle \mathbf{w}_i, \mathbf{v}_j \rangle |^2 \geq d_W,
    \end{equation*}
where $d_W$ is the dimension of $W$. Moreover, equality holds if and only if $\{\mathbf{v}_j\}_{j=1}^N $ is the canonical oblique dual $ \{\bpi_{VW^\perp} \mathbf{S}^{\dagger} \mathbf{w}_j\}_{j=1}^N$.   We further show that the mixed coherence between $\{\mathbf{w}_i\}_{i=1}^N$ and $\{\mathbf{v}_i\}_{i=1}^N$ satisfies
    \begin{equation*}
       \underset{i \neq j}{\text{max}} \ |\langle \mathbf{w}_i, \mathbf{v}_j \rangle |^2 \geq \frac{d_W(N-d_W)}{N^2(N-1)},
    \end{equation*}
where saturation requires the existence of an $(N,d_W)$-equiangular tight frame.

In \cref{section:probabilisticdual}, we introduce the notion of oblique dual probabilistic frames where $\mathcal{H} = \mathbb{R}^n$ such that $\mathbb{R}^n = W \oplus V^\perp$. Recall that $\mu \in  \mathcal{P}_2(W)$ is called a \emph{probabilistic frame} for $W$ if there exist $0<A \leq B < \infty$ such that for any ${\bf x} \in W$, 
\begin{equation*}
    A \Vert \mathbf{x}  \Vert^2 \leq\int_{W} \vert \left\langle \mathbf{x},\mathbf{y} \right\rangle \vert ^2 d\mu(\mathbf{y})  \leq B \Vert \mathbf{x}  \Vert^2.
\end{equation*}
Then we will say $\nu \in \mathcal{P}_2(V)$ is an \emph{oblique dual probabilistic frame} of $\mu$ on $V$ if there exists a transport coupling $\gamma \in \Gamma(\mu, \nu)$ such that
    \begin{equation*}
       \bpi_{WV^\perp} = \int_{W \times V} \mathbf{x} \mathbf{y}^t d\gamma(\mathbf{x}, \mathbf{y}),
    \end{equation*}
where $\bpi_{WV^\perp}$ is the oblique projection of $\mathbb{R}^n$ onto $W$ along $V^\perp$. 
Furthermore, given $\epsilon \geq 0$,
$\nu \in \mathcal{P}_2(V)$ is called an \textit{oblique $\epsilon$-approximately dual probabilistic frame} of $\mu$ on $V$ if there exists $\gamma \in \Gamma(\mu, \nu)$ such that
    \begin{equation*}
        \left \| \int_{W \times V} \mathbf{x} \mathbf{y}^t d\gamma(\mathbf{x}, \mathbf{y}) -\bpi_{WV^\perp}  \right \| \leq \epsilon.
    \end{equation*}
    
In \cref{section:probabilisticDualPotential}, we introduce the oblique dual probabilistic frame potential. In particular, in \cref{dualframepotenial2}, we show that if $\mu \in \mathcal{P}_2(W)$ is a probabilistic frame for $W$ with bounds $0<A \leq B$, and  $\nu \in \mathcal{P}_2(V)$ is an oblique dual of $\mu$ on $V$, then
\begin{equation*}
     \int_{W}  \int_{V} |\langle \mathbf{x}, \mathbf{y} \rangle |^2 d\mu(\mathbf{x}) d\nu(\mathbf{y}) \geq \frac{A}{B}d_W,
\end{equation*}
and equality holds if and only if $\mu$ is a tight probabilistic frame for $W$ with bound $A>0$ and $\nu $ is the canonical oblique dual $ ({\bpi_{VW^\perp}\mathbf{S}^{\dagger}_\mu})_\# \mu$. Furthermore, as shown in \cref{DualFramePotential}, the tightness assumption can be dropped if we minimize among oblique dual frames of pushforward type.

Finally, in \cref{section:approximateDual}, we study oblique $\epsilon$-approximately dual probabilistic frames. We show that if a probability measure is sufficiently close to an oblique dual pair in the 2-Wasserstein metric, then it constitutes an oblique $\epsilon$-approximately dual probabilistic frame. In particular, \cref{interiorPoint} shows that given a probabilistic frame, its oblique dual probabilistic frames are interior points in the set of oblique $\epsilon$-approximately dual frames in the $2$-Wasserstein topology.

\section{Preliminaries}\label{section:preliminaries}
Throughout the paper, $\mathcal{H}$ is a separable Hilbert space, and $W$ and $V$ are closed subspaces of $\mathcal{H}$ such that $\mathcal{H} = W \oplus V^\perp$, where $V^\perp$ is the orthogonal complement of $V$. In addition, 
$\bP_W$ is the orthogonal projection of $\mathcal{H}$ onto $W$, $\mathbf{0}$ the zero vector in $\mathcal{H}$, $\mathbf{0}_{n \times n}$ the zero matrix of size $n \times n$, and $\mathbf{Id}$ the identity matrix of size $n \times n$. We also use ${\bf x}^t$ to denote the transpose of a vector ${\bf x} \in \mathbb{R}^n$. 

\subsection{Oblique Projection, Oblique Dual Frame, and Consistent Reconstruction}

The following theorem gives some alternate characterizations of the direct sum assumption $\mathcal{H} = W \oplus V^{\perp}$ and
relies on the concept of the maximum angle $\theta_{WV} \in [0, \frac{\pi}{2}]$ between the subspaces $W$ and $V$ given by
\begin{equation*}
    \cos(\theta_{WV}) := \inf_{{\mathbf{f} \in W, \,\|\mathbf{f}\|=1}} \| \bP_V \mathbf{f} \|. 
\end{equation*}
\begin{theorem}[{\cite[Theorem 2.3]{tang2000oblique}}]\label{lem:directsumangle}
Suppose $W$ and $V$ are closed subspaces of a separable Hilbert space $\mathcal{H}$. Then the following are equivalent:
\begin{itemize}
  \item[$(1)$] $\mathcal{H} = W \oplus V^{\perp}$.
    \item[$(2)$] $\mathcal{H} = V \oplus W^{\perp}$.
    \item[$(3)$] $\cos(\theta_{VW}) > 0$ and $\cos(\theta_{WV}) > 0$.
    \item[$(4)$] There exist Riesz bases $\{\mathbf{w}_i\}_{i \in I}$ and $\{\mathbf{v}_i\}_{i \in I}$ for
    $W$ and $V$, respectively, such that $\{\mathbf{w}_i\}_{i \in I}$ is biorthogonal to $\{\mathbf{v}_i\}_{i \in I}$.
\end{itemize}
\end{theorem}

We then recall the definition of oblique projection, which can also be used to define oblique dual frames.  

\begin{definition}
    Suppose $\mathcal{H} = W \oplus V^\perp$. Then $\bpi_{WV^\perp}: \mathcal{H} \rightarrow W$ is called the \emph{oblique projection} of $\mathcal{H}$ onto $W$ along $V^\perp$ if for any $\mathbf{w} \in W$ and any $\mathbf{v} \in V^\perp$, 
    $$\bpi_{WV^\perp} \mathbf{w} = \mathbf{w} \ \text{and} \ \bpi_{WV^\perp} \mathbf{v} = \mathbf{0}.$$
\end{definition}

One can check that $\bpi_{WV^\perp}$ is the adjoint of $\bpi_{VW^\perp}$, i.e., $\bpi_{WV^\perp}= (\bpi_{VW^\perp})^*$. Note also that $\bpi_{WW^\perp} = \bP_W$, the orthogonal projection onto $W$. Moreover, the following lemma gives basic identities for oblique projections.
\begin{lemma}[{\cite[Lemma 2.1]{heineken2018oblique}}]\label{lemma:oblique_proj_identity}
  Let $W$ and $V$ be closed subspaces of $\mathcal{H}$ so that $\mathcal{H} = W \oplus V^\perp$. Then $\bpi_{VW^\perp} \bP_{W} = \bpi_{VW^\perp}$ and $\bP_{W} \bpi_{VW^\perp} = \bP_{W}$. 
\end{lemma}
\begin{proof}
    Since $W^\perp \subset \kernel (\bpi_{VW^\perp})$, we see that 
    \[
        \bpi_{VW^\perp} \bP_{W} = \bpi_{VW^\perp} (\bP_{W}+\bP_{W^\perp})= \bpi_{VW^\perp}.
    \]
    Similarly, 
    \[
        \bP_{W} \bpi_{VW^\perp} =\bP_{W} (\bpi_{VW^\perp}+ \bpi_{W^\perp V}) = \bP_{W}.
    \]
\end{proof}

Note that if $\{\mathbf{v}_i\}_{i\in I}$ is an oblique dual frame of $\{\mathbf{w}_i\}_{i\in I}$ on $V$, then $\{\mathbf{w}_i\}_{i\in I}$ and $\{\bP_W \mathbf{v}_i\}_{i\in I}$ are dual frames for $W$; similarly,
$\{\mathbf{v}_i\}_{i\in I}$ and $\{\bP_V \mathbf{w}_i\}_{i\in I}$ are dual frames for $V$ \cite{christensen2004oblique}.
We are now ready to provide a few equivalent rigorous definitions of oblique dual frames.

\begin{definition}[{\cite[Lemma 3.1]{christensen2004oblique}}]\label{def:obliquedual}
Suppose $\{\mathbf{w}_i\}_{i\in I}$ and $\{\mathbf{v}_i\}_{i\in I}$ are frames for $W$ and $V$, respectively, where $\mathcal{H} = W \oplus V^\perp$.  Then $\{\mathbf{v}_i\}_{i\in I}$ is said to be an \emph{oblique dual frame} of $\{\mathbf{w}_i\}_{i\in I}$ on $V$ if one of the following equivalent conditions holds:

\begin{enumerate}
    \item \label{it:perfect reconstruction} $\mathbf{f} = \sum_{i\in I} \langle \mathbf{f}, \mathbf{v}_i \rangle \mathbf{w}_i, \ \text{for any } \mathbf{f} \in W$. 

    \item $\bpi_{WV^\perp} \mathbf{f} = \sum_{i\in I} \langle \mathbf{f}, \mathbf{v}_i \rangle \mathbf{w}_i, \ \text{for any } \mathbf{f} \in \mathcal{H}$. 

    \item $\bpi_{VW^\perp} \mathbf{f} = \sum_{i\in I} \langle \mathbf{f}, \mathbf{w}_i \rangle \mathbf{v}_i, \ \text{for any } \mathbf{f} \in \mathcal{H}$. 

     \item $\langle \bpi_{WV^\perp} \mathbf{f}, \mathbf{g} \rangle = \sum_{i\in I} \langle \mathbf{f}, \mathbf{v}_i \rangle \langle \mathbf{w}_i, \mathbf{g} \rangle, \ \text{for any } \mathbf{f}, \mathbf{g} \in \mathcal{H}$. 

    \item $\langle \bpi_{VW^\perp} \mathbf{f}, \mathbf{g} \rangle = \sum_{i\in I} \langle \mathbf{f}, \mathbf{w}_i \rangle \langle \mathbf{v}_i, \mathbf{g} \rangle, \ \text{for any } \mathbf{f}, \mathbf{g} \in \mathcal{H}$. 
\end{enumerate}
In this case, $\{\mathbf{v}_i\}_{i\in I}$ and $\{\mathbf{w}_i\}_{i\in I}$  are called an \emph{oblique dual pair}.
\end{definition}

The notion of consistent reconstruction, which was first introduced in \cite{unser2002general}, provides another motivation for studying oblique dual frames.

\begin{definition}\label{def:consistRecons}
 Let $\{\mathbf{w}_i\}_{i\in I}$ and $\{\mathbf{v}_i\}_{i\in I}$ be frames for $W$ and $V$, respectively,  where $W$ and $V$ are closed subspaces of $\cH$ and $W \cap V^\perp =\{\bf 0\}$. Then $\{\mathbf{w}_i\}_{i\in I}$ and $\{\mathbf{v}_i\}_{i\in I}$ are said to perform \emph{consistent reconstruction} if, for any $\mathbf{f} \in \mathcal{H}$, $\langle \mathbf{f}, \mathbf{v}_i \rangle = \langle \hat{\mathbf{f}}, \mathbf{v}_i \rangle$ for each $i$. Here, $\hat{\mathbf{f}} = \sum_{i \in I} \langle \mathbf{f}, \mathbf{v}_i \rangle \mathbf{w}_i$ is the signal reconstructed from $\mathbf{f}$. 
\end{definition}

The assumption that $W \cap V^\perp =\{0\}$ in \cref{def:consistRecons} is there, in part, to ensure unique reconstruction: that is, for any $\mathbf{f}, \mathbf{g} \in W$, if $\langle \mathbf{f}, \mathbf{v}_i\rangle = \langle \mathbf{g}, \mathbf{v}_i\rangle$ for all $i$,
then $\mathbf{f}=\mathbf{g}$. Indeed, if $W \cap V^\perp = \{\mathbf{0}\}$ and $\langle \mathbf{f}, \mathbf{v}_i \rangle = \langle \mathbf{g}, \mathbf{v}_i \rangle$ for all $i$, then $\mathbf{f-g} \in V^\perp$ since $\{\mathbf{v}_i\}_{i \in I}$ is a frame for $V$. Hence, $\mathbf{f-g} \in W \cap V^\perp = \{\mathbf{0}\}$ and thus $\mathbf{f} = \mathbf{g}$. Conversely, if there is some nonzero $\mathbf{h} \in W \cap V^\perp$, then for any $\mathbf{f} \in W$, the distinct vectors $\mathbf{f}$ and $\mathbf{g} := \mathbf{f}+\mathbf{h} \in W$ produce identical samples: $\langle \mathbf{f}, \mathbf{v}_i\rangle = \langle \mathbf{g}, \mathbf{v}_i\rangle$ for all $i$.

Note that consistent reconstruction implies that $\cH = W \oplus V^\perp$. To see this, suppose $\{\mathbf{w}_i\}_{i \in I}$ and $\{\mathbf{v}_i\}_{i \in I}$ perform consistent reconstruction on $\mathcal{H}$. Then for every $ \mathbf{f} \in \mathcal{H}$, we know that $\langle \mathbf{f}, \mathbf{v}_i \rangle = \langle \hat{\mathbf{f}}, \mathbf{v}_i \rangle$ for all $i \in I$. Since $\{\mathbf{v}_i\}_{i \in I}$ is a frame for $V$, this implies $\mathbf{f}-\hat{\mathbf{f}} \in V^\perp$. This means that every $\mathbf{f} \in \mathcal{H}$ admits the decomposition
\[
    \mathbf{f} = \hat{\mathbf{f}} + (\mathbf{f}-\hat{\mathbf{f}}) \in W \oplus V^\perp.
\]
Since $W \cap V^\perp = \{\mathbf{0}\}$, this implies $\mathcal{H} = W \oplus V^\perp$.

The following theorem formalizes the connection between consistent reconstruction and oblique dual frames. The relation between oblique projection and consistent reconstruction is also illustrated in \cref{fig:sample-image}. 

\begin{theorem}[{\cite[Theorem 1]{eldar2003sampling} and \cite[Lemma 3.1]{christensen2004oblique}}] \label{thm:Oblique_dual_characterization}
    Let $\{\mathbf{w}_i\}_{i\in I}$ and $\{\mathbf{v}_i\}_{i\in I}$ be frames for $W$ and $V$ where $\mathcal{H} = W \oplus V^\perp$. Then $\{\mathbf{w}_i\}_{i\in I}$ and $\{\mathbf{v}_i\}_{i\in I}$ perform consistent reconstruction if and only if $\{\mathbf{v}_i\}_{i\in I}$ is an \emph{oblique dual frame} of $\{\mathbf{w}_i\}_{i\in I}$ on $V$: for any $\mathbf{f} \in \mathcal{H}$, $\bpi_{W V^\perp} \mathbf{f} = \hat{\mathbf{f}}$, where $ \hat{\mathbf{f}}= \sum_{i\in I} \langle \mathbf{f}, \mathbf{v}_i \rangle \mathbf{w}_i$.
\end{theorem}

Furthermore, O.~Christensen and Y.~C.~Eldar \cite{christensen2004oblique} provided a parameterization of all oblique dual frames of a fixed frame for $W$.
    \begin{theorem}[{\cite[Theorem 3.2]{christensen2004oblique}}]\label{Oblique_Characterization}
   Let $\mathcal{H} = W \oplus V^\perp$ and let $\{\mathbf{w}_i\}_{i\in I}$ be a frame for $W$. Then the oblique dual frames of $\{\mathbf{w}_i\}_{i\in I}$ on $V$  are precisely the families
\begin{equation*}
    \left\{\mathbf{v}_i\right\}_{i\in I} = \left\{ \bpi_{VW^\perp} \mathbf{S}^\dagger \mathbf{w}_i + \mathbf{h}_i - \sum_{j\in I} \langle \mathbf{S}^\dagger \mathbf{w}_i, \mathbf{w}_j \rangle \mathbf{h}_j \right\}_{i\in I},
\end{equation*}
where $\{\mathbf{h}_i\}_{i\in I} \subset V$ is a Bessel sequence for $\mathcal{H}$ and $\mathbf{S}^\dagger$ is the Moore–Penrose inverse of the frame operator $\mathbf{S}$ for $\{\mathbf{w}_i\}_{i\in I}$. 
\end{theorem}

The frame $\{\bpi_{VW^\perp} \mathbf{S}^\dagger \mathbf{w}_i\}_{i\in I}$ is called the \emph{canonical oblique dual} of $\{\mathbf{w}_i\}_{i\in I}$. It was also shown in \cite{eldar2003sampling,christensen2004oblique,eldar2005general} that for a given $\mathbf{f} \in \mathcal{H}$, among all coefficients 
$\{c_i\}_{i\in I} \in \ell^2(I)$ for which $
\bpi_{WV^\perp} \mathbf{f} = \sum_{i\in I} c_i \mathbf{w}_i$,
the coefficient sequence with the minimal $\ell^2$ energy is given by $\{\langle \mathbf{f}, \bpi_{VW^\perp} \mathbf{S}^\dagger \mathbf{w}_i\rangle\}_{i\in I}$.
The reconstruction error using oblique dual frames for signal processing can be found in \cite{unser2002general, unser2002generalized, eldar2003sampling, berger2019sampling}: for any $\mathbf{f} \in \mathcal{H}$, 
\begin{equation*}
\| \mathbf{f} - \bP_{W} \mathbf{f} \| \leq \| \mathbf{f} - \bpi_{WV^\perp} \mathbf{f} \|
\leq \frac{1}{\cos(\theta_{WV})} \| \mathbf{f} - \bP_{W} \mathbf{f} \|,
\end{equation*}
where $\theta_{WV} \in [0, \frac{\pi}{2}]$ is the maximum angle between subspaces $W$ and $V$.
This inequality quantifies the stability and loss introduced by using distinct sampling and reconstruction subspaces. Moreover, it shows that the reconstruction via $\bpi_{W V^\perp}$ and oblique dual frames is at least as accurate as the orthogonal projection onto $W$, up to a factor depending on the angle between $W$ and $V$.

The following is a simple example of oblique dual frames in $\mathbb{R}^2$.
\begin{example}
Suppose $\mathbf{x} = \begin{pmatrix}
        1 \\0
    \end{pmatrix}$, $\mathbf{y} = \begin{pmatrix}
        1 \\1
    \end{pmatrix}$, and $\mathbf{z} = \begin{pmatrix}
        1 \\-1
    \end{pmatrix}$.
    Let $W = \spanning\{\mathbf{x}\}$ and $V = \spanning\{\mathbf{y}\}$. Then $V^\perp = \spanning\{\mathbf{z}\}$ and $\mathbb{R}^2 = W \oplus V^\perp$. Clearly, $\{\mathbf{x}\}$ is a frame for $W$ and $\{\mathbf{y}\}$ is a frame for $V$. And for any $\mathbf{f} = \begin{pmatrix}
        a \\0
    \end{pmatrix} \in W$ where $a \in \mathbb{R}$, 
    \begin{equation*}
        \mathbf{f} = a \mathbf{x}= \langle \mathbf{f}, \mathbf{y} \rangle \mathbf{x}.
    \end{equation*}
Thus, $\{\mathbf{y}\}$ is an oblique dual frame of $\{\mathbf{x}\}$ on $V$. 
\end{example}

\begin{figure}[!ht]
    \centering
    \includegraphics[scale=0.7]{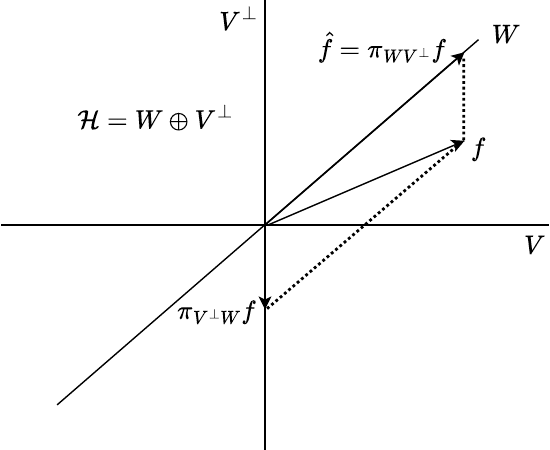} 
    \caption{Illustration of oblique projections and consistent reconstruction. Here $W$ and $V$ are closed subspaces of a Hilbert space $\mathcal{H}$ with $\mathcal{H} = W \oplus V^\perp$. $\{\mathbf{w}_i\}_{i \in I} \subset W$ and $\{\mathbf{v}_i\}_{i \in I} \subset V$ are not only Bessel sequences for $\mathcal{H}$, but $\{\mathbf{w}_i\}_{i \in I}$ and $\{\mathbf{v}_i\}_{i \in I}$ are frames for $W$ and $V$, respectively. Furthermore, $\bpi_{WV^\perp}$ is the oblique projection of $\mathcal{H}$ onto $W$ along $V^\perp$. Consistent reconstruction requires that $\bpi_{W V^\perp}\mathbf{f} = \hat{\mathbf{f}}$ for all $\mathbf{f} \in \mathcal{H}$, where $\hat{\mathbf{f}} = \sum_{i\in I} \langle \mathbf{f}, \mathbf{v}_i \rangle \mathbf{w}_i$.} 
    \label{fig:sample-image}
\end{figure}

\subsection{Probabilistic Frames and Optimal Transport} 
Let $\mathcal{P}(\mathbb{R}^n)$ be the set of Borel probability measures on $\mathbb{R}^n$ and $\mathcal{P}_2(\mathbb{R}^n) \subset \mathcal{P}(\mathbb{R}^n)$  the set of Borel probability measures with finite second moments. That is, if $\mu \in \mathcal{P}_2(\mathbb{R}^n)$, then its second moment satisfies
\begin{equation*}
   M_2(\mu) := \int_{\mathbb{R}^n} \Vert {\bf x} \Vert ^2 d\mu({\bf x}) < + \infty .
\end{equation*}
The \textit{support} of $\mu \in  \mathcal{P}(\mathbb{R}^n)$ is defined by 
$$\supp(\mu) = \left\{ {\bf x} \in \mathbb{R}^n: \text{for any} \ r>0, \,\mu(B_r({\bf x}))>0  \right \},$$
where $B_r({\bf x})$ is the open ball centered at ${\bf x}$ with radius $r>0$.
If $\mu \in  \mathcal{P}(\mathbb{R}^n) $ and $f: \mathbb{R}^n \rightarrow \mathbb{R}^m$ is a Borel measurable map where $n$ may differ from $m$, then $f_{\#} \mu \in \mathcal{P}(\mathbb{R}^m)$ is called the \textit{pushforward} of $\mu$ by  the map $f$,  and is defined as
$$f_{\#} \mu (E) := (\mu \circ f^{-1}) (E) = \mu \big (f^{-1}(E) \big ) \ \text{for any Borel set}  \ E \subset \mathbb{R}^m.$$ 
If $f$ is linear and represented by a matrix ${\bf A}$ with respect to some basis, then ${\bf A}_{\#} \mu$ is used to denote $f_{\#} \mu$. In particular, $(\mathbf{Id}, f)$ is used to denote the map from $\mathbb{R}^n$ to $\mathbb{R}^n \times \mathbb{R}^m$ via $\mathbf{x} \mapsto (\mathbf{x}, f(\mathbf{x}))$. In such a case, $(\mathbf{Id}, f)_\#\mu$ is a probability measure on $\mathbb{R}^n \times \mathbb{R}^m$ that is supported on the graph of $f$.

Finite frames can be viewed as discrete probability measures. Consequently, one can use probability tools (namely, optimal transport) to study frames. Suppose $\{{\bf x}_i\}_{i=1}^N$ is a finite frame with bounds $0<A \leq B$ for a subspace $W$ in the Euclidean space $\mathbb{R}^n$. Letting $\mu_f := \frac{1}{N} \sum\limits_{i=1}^N   \delta_{{\bf x}_i} \in \mathcal{P}(W)$,  the frame definition in~\cref{eq:frame} becomes 
\begin{equation*}
    \frac{A}{N} \Vert {\bf x}  \Vert^2 \leq\int_{W} \vert \left\langle {\bf x},{\bf y} \right\rangle \vert ^2 d\mu_f({\bf y})  \leq \frac{B}{N} \Vert {\bf x}  \Vert^2, \ \text{for any} \ {\bf x} \in W.
\end{equation*}
Inspired by this observation, Ehler~\cite{ehler2012random} introduced the concept of probabilistic frames. Subsequently, Ehler and Okoudjou investigated the probabilistic frame potential \cite{ehler2012minimization} and surveyed this area with particular emphasis on its connections to optimal transport \cite{ehler2013probabilistic}. 
Note that when $W \subset \mathbb{R}^n$ is closed (e.g., a subspace), $\mu \in \mathcal{P}(W)$ means that we can extend the measure to a measure in $\mathcal{P}(\mathbb{R}^n)$ by setting the measure of the open set $\bR^n  \setminus W$ to be zero.  We will denote this extended measure also with $\mu$.

\begin{definition}
Let $W$ be a subspace of $\mathbb{R}^n$. $\mu \in  \mathcal{P}(W)$ is called a \emph{probabilistic frame} for $W$ if there exist $0<A \leq B < \infty$ such that for any ${\bf x} \in W$, 
\begin{equation*}
    A \Vert {\bf x}  \Vert^2 \leq\int_{W} \vert \left\langle {\bf x},{\bf y} \right\rangle \vert ^2 d\mu({\bf y})  \leq B \Vert {\bf x}  \Vert^2.
\end{equation*}
The probabilistic frame $\mu$ is said to be \emph{tight} if we may choose $A = B$ and \emph{Parseval} if $A = B=1$. Moreover, $\mu$ is a \emph{Bessel probability measure} for $W$ if the upper (but possibly not the lower) bound holds. 
\end{definition}

One can also define the \emph{frame operator} ${\bf S}_\mu$ for a probabilistic frame $\mu$ on $W$ as the following symmetric positive semi-definite matrix
    \begin{equation*}
          {\bf S}_\mu := \int_{W}  {\bf y} {\bf y}^t d \mu({\bf y}).
    \end{equation*}
Probabilistic frames can be characterized by their frame operators. 

 \begin{proposition}[{\cite[Theorem 12.1]{ehler2013probabilistic} and \cite[Proposition 3.1]{maslouhi2019probabilistic}}] \label{TAcharacterization}
Let $\mu  \in \mathcal{P}(W)$. Then $\mu$ is a probabilistic frame on $W$ if and only if $\mu \in \mathcal{P}_2(W)$ and $\spanning\{\supp(\mu)\} = W$. Furthermore, if $\mu  \in \mathcal{P}_2(W)$, then $\mu$ is a probabilistic frame on $W$ if and only if the restriction of ${\bf S}_{\mu}$ to $W$ is positive definite, and $\mu$ is a tight probabilistic frame on $W$ with bound $A > 0$ if and only if ${\bf S}_{\mu} = A \ \bP_W$, where $\bP_W$ is the orthogonal projection of $\mathbb{R}^n$ onto $W$. 
\end{proposition}

Since probabilistic frames on $W$ have finite second moments, they naturally lie in $\mathcal{P}_2(W)$ and can be studied using optimal transport and the 2-Wasserstein distance~\cite{ehler2013probabilistic, wickman2014optimal}. Given two probabilistic frames $\mu$ and $\nu$ on $W$, suppose $\Gamma(\mu,\nu)$ is the set of transport couplings with marginals $\mu$ and $\nu$, that is,
$$ \Gamma(\mu,\nu) :=  \left \{ \gamma \in \mathcal{P}(W \times W): {\pi_{{ x}}}_{\#} \gamma = \mu, \ {\pi_{{ y}}}_{\#} \gamma = \nu \right \},$$
where $\pi_{{ x}}$, $\pi_{{ y}}$ are projections onto the ${\bf x}$ and ${\bf y}$ coordinates: for any $({\bf x},{\bf y}) \in \mathbb{R}^n \times \mathbb{R}^n $, $\pi_{{ x}}({\bf x}, {\bf y}) = {\bf x}$ and $ \pi_{{ y}}({\bf x},{\bf y}) = {\bf y}$. The 2-Wasserstein distance $ W_2(\mu,\nu)$ is often used to quantify the distance between $\mu$ and $\nu$:  
$$ W_2(\mu,\nu) := \left(\underset{\gamma \in \Gamma(\mu,\nu)}{\operatorname{inf}}  \int_{W \times W}  \left \Vert {\bf x}-{\bf y}\right \Vert^2 \ d\gamma({\bf x},{\bf y})\right)^{1/2}.$$
This is of interest even if one is only concerned with finite frames, since the $2$-Wasserstein distance can quantify the distance between probabilistic frames induced by finite frames of different cardinalities.
See~\cite{figalli2021invitation} for details on optimal transport and Wasserstein distance.

Let $W$ be a subspace of $\bR^n$ and $\mu \in \mathcal{P}_2(W)$. As above, consider the extension of $\mu$---which we also call $\mu$---to $\cP(\bR^n)$ by setting the measure of $\bR^n \setminus W$ to be zero. We may then extend any measurable $f: W \rar \bR$ to $\bR^n \rar \bR$ by zero-padding or restrict any $f: \bR^n \rar \bR$ to $W \rar \bR$ to see that $L^2(\mu, W) \cong L^2(\mu,\bR^n)$. Furthermore, note that $\mu \in \mathcal{P}_2(W)$ also implies $\mu \in \mathcal{P}_2(\mathbb{R}^n)$, since
\begin{equation*}
   \int_{\mathbb{R}^n} \Vert {\bf x} \Vert ^2 d\mu({\bf x}) = \int_{W} \Vert {\bf x} \Vert ^2 d\mu({\bf x})< + \infty .
\end{equation*}

Therefore, for $\mu \in \mathcal{P}_2(W)$, its analysis operator $U_\mu: \mathbb{R}^n \rightarrow L^2(\mu, W)$ given by $(U_\mu{\bf x})(\cdot) = \langle {\bf x}, \cdot \rangle \in L^2(\mu, W)$ is well-defined (bounded), because for any ${\bf x} \in \mathbb{R}^n$, 
\begin{equation*}
    \| U_\mu{\bf x}\|^2 = \int_{W} |\langle{\bf x}, {\bf y}\rangle|^2 d\mu({\bf y}) \leq \|{\bf x}\|^2 M_2(\mu). 
\end{equation*}
Then the adjoint (synthesis) operator $U_\mu^*: L^2(\mu, W) \rightarrow \mathbb{R}^n$ exists and is given by 
\begin{equation*}
    \ U_\mu^*(\psi) = \int_{W} {\bf x}  \psi({\bf x}) d\mu({\bf x}). 
\end{equation*}
One can verify that as a map on $\mathbb{R}^n$, the frame operator satisfies ${\bf S}_\mu =  U_\mu^* U_\mu$.

Similar to frames, one can reconstruct signals using \emph{dual probabilistic frames}, which are also known as transport duals \cite{wickman2014optimal}. 

\begin{definition}\label{dualDefinition}
Suppose $\mu$ is a probabilistic frame for $W$. $\nu \in \mathcal{P}_2(W)$ is called a \emph{dual probabilistic frame} of $\mu$ on $W$ if there exists $\gamma \in \Gamma(\mu, \nu)$ such that 
\begin{equation*}
   \int_{W \times W} {\bf x}{\bf y}^t d\gamma({\bf x}, {\bf y}) = \bP_W, 
\end{equation*}
which is equivalent to the statement that for any ${\bf f} \in W$, $\int_{W \times W} {\bf x} \langle {\bf y}, {\bf f} \rangle d\gamma({\bf x}, {\bf y}) = {\bf f}$. 
\end{definition}

Note that $\nu \in \mathcal{P}_2(V)$ automatically implies (via extension) that $\nu \in \mathcal{P}_2(\mathbb{R}^n)$, and one can still define $ \Gamma(\mu,\nu) :=  \{ \gamma \in \mathcal{P}(W \times V): {\pi_{{ x}}}_{\#} \gamma = \mu, \ {\pi_{{ y}}}_{\#} \gamma = \nu \}$. In addition, if $\gamma \in \Gamma(\mu, \nu)$, then $\gamma \in  \mathcal{P}_2(W \times V)$ also implies $\gamma \in  \mathcal{P}_2(\mathbb{R}^n \times \mathbb{R}^n)$, that is, 
\begin{equation*}
  \int_{\mathbb{R}^n \times \mathbb{R}^n} \mathbf{x} \mathbf{y}^t d\gamma(\mathbf{x}, \mathbf{y}) =  \int_{W \times V} \mathbf{x} \mathbf{y}^t d\gamma(\mathbf{x}, \mathbf{y}).
\end{equation*}

Let $\mu$ be a probabilistic frame for $W$ and ${\bf S}_{\mu}^\dagger$ the Moore–Penrose inverse of ${\bf S}_{\mu}$. Then ${{\bf S}_{\mu}^\dagger}_\#\mu$ is called the canonical dual frame of $\mu$ on $W$, since 
\begin{equation*}
    \int_{W \times W} {\bf x}{\bf y}^t d(\mathbf{Id}, {\bf S}_{\mu}^\dagger)_\#\mu({\bf x,y}) = \int_{W} {\bf x}{\bf x}^t d\mu({\bf x}) \ {\bf S}_{\mu}^\dagger = {\bf S}_{\mu} {\bf S}_{\mu}^\dagger= \bP_W,
\end{equation*}
where the last identity follows from \cref{lemma:frame_projection}.
Similarly,  ${({\bf S}_{\mu}^\dagger)^{\frac{1}{2}}}_\#\mu$ is called the canonical Parseval  frame of $\mu$ on $W$, since
\begin{equation*}
    \int_{W} {\bf x}{\bf x}^t d{({\bf S}_{\mu}^\dagger)^{\frac{1}{2}}}_\#\mu({\bf x}) =({\bf S}_{\mu}^\dagger)^{1/2} {\bf S}_{\mu}  ({\bf S}_{\mu}^\dagger)^{1/2}= \bP_W.
\end{equation*}

\begin{lemma}\label{lemma:frame_projection}
    Let $W$ be a subspace in $\mathbb{R}^n$ and $\mu$ be a probabilistic frame for $W$ with frame operator ${\bf S}_\mu$. Then $\mathrm{Ran}({\bf S}_\mu)=W$ and $\mathrm{Ker}({\bf S}_\mu) = W^\perp$. Furthermore, ${\bf S}_\mu {\bf S}_\mu^{\dagger}$, ${\bf S}_\mu^{\dagger} {\bf S}_\mu$, and $({\bf S}_{\mu}^\dagger)^{1/2} {\bf S}_{\mu}  ({\bf S}_{\mu}^\dagger)^{1/2}$ are orthogonal projections of $\mathbb{R}^n$ onto the subspace $W$. 
\end{lemma}
\begin{proof}
By standard properties of the Moore–Penrose inverse, ${\bf S}_\mu {\bf S}_\mu^{\dagger}$ is the orthogonal projection onto the range of ${\bf S}_\mu$. Since $\mu$ is a probabilistic frame for $W$, then by \cref{TAcharacterization},  $\mu$ is supported on $W$. Then for any $ \mathbf{f} \in \mathbb{R}^n$,
  $$
  {\bf S}_\mu \mathbf{f}=\int_W \langle \mathbf{f},\mathbf{x}\rangle \mathbf{x} d\mu(\mathbf{x}) \in W.
  $$
That is to say, $\mathrm{Ran}({\bf S}_\mu)\subseteq W$. Also note that $W^\perp \subseteq \mathrm{Ker}({\bf S}_\mu)$, since for any $\mathbf{f} \in W^\perp$, ${\bf S}_\mu \mathbf{f} = \mathbf{0}$. If we can show that $\mathrm{Ker}({\bf S}_\mu) = W^\perp$, then it would follow from the Rank–Nullity theorem that
\begin{equation*}
    \Dim(\mathrm{Ran}({\bf S}_\mu)) + \Dim(W^\perp) = n,
\end{equation*}
and hence that
$\mathrm{Ran}({\bf S}_\mu)=W$. 

Now let us show $\mathrm{Ker}({\bf S}_\mu) = W^\perp$. Suppose $\mathbf{f} \in \operatorname{Ker}({\bf S}_\mu)$. Since $\mathbf{f} = \mathbf{f}_W + \mathbf{f}_{W^\perp}$, where $\mathbf{f}_W  \in W$ and $\mathbf{f}_{W^\perp} \in W^\perp$, we have ${\bf S}_\mu \mathbf{f}_{W^\perp} = \mathbf{0}$ and hence 
\begin{equation*}
  \mathbf{0} = {\bf S}_\mu \mathbf{f} =  {\bf S}_\mu (\mathbf{f}_W + \mathbf{f}_{W^\perp}) = {\bf S}_\mu \mathbf{f}_W.
\end{equation*}
Since $\mu$ is a probabilistic frame for $W$, the restriction of ${\bf S}_\mu$  on $W$ is injective by \cref{TAcharacterization}, so $\mathbf{0} = {\bf S}_\mu \mathbf{f}_W$ implies $\mathbf{f}_W = \mathbf{0}$. Therefore, $\mathbf{f} =  \mathbf{f}_{W^\perp} \in W^\perp$ and we conclude that $\mathrm{Ker}({\bf S}_\mu) = W^\perp$.

We have now shown that $\mathrm{Ran}({\bf S}_\mu)=W$, 
and thus ${\bf S}_\mu {\bf S}_\mu^{\dagger}$ is the orthogonal projection of $\mathbb{R}^n$ onto $W$. Since ${\bf S}_\mu$ is symmetric, it follows that ${\bf S}_\mu^{\dagger} {\bf S}_\mu$ is also the orthogonal projection onto $W$. Finally, the spectral decomposition of ${\bf S}_\mu$ shows that $({\bf S}_{\mu}^\dagger)^{1/2} {\bf S}_{\mu}  ({\bf S}_{\mu}^\dagger)^{1/2}$ is the orthogonal projection of $\mathbb{R}^n$ onto the subspace $W$. 
\end{proof}

We also need the following gluing lemma, a standard tool in optimal transport, to ``glue'' two transport couplings together. Analogously to $\pi_x$ and $\pi_y$, let $\pi_{xy}$ and $\pi_{yz}$ be the orthogonal projections onto the corresponding coordinates, i.e.,  for any $({\bf x}, {\bf y}, {\bf z}) \in \mathbb{R}^n \times \mathbb{R}^n \times \mathbb{R}^n$, 
\begin{equation*}
    \pi_x({\bf x}, {\bf y}, {\bf z}) = {\bf x}, \ \pi_y({\bf x}, {\bf y}, {\bf z}) = {\bf y}, \ \pi_{xy}({\bf x}, {\bf y}, {\bf z})=({\bf x}, {\bf y}), \ \pi_{yz}({\bf x}, {\bf y}, {\bf z})=({\bf y}, {\bf z}).
\end{equation*}

\begin{lemma}[Gluing Lemma {\cite[p.~59]{figalli2021invitation}}]\label{gluinglemma} 
Let $W_1$, $W_2$, and $W_3$ be subspaces of $\mathbb{R}^n$ and $\mu_1 \in \mathcal{P}_2(W_1)$, $\mu_2\in \mathcal{P}_2(W_2)$, and $\mu_3 \in \mathcal{P}_2(W_3)$, respectively. Suppose $\gamma^{12} \in \Gamma(\mu_1, \mu_2) \subset \mathcal{P}(W_1 \times W_2)$ and $\gamma^{23} \in \Gamma(\mu_2, \mu_3) \subset \mathcal{P}(W_2 \times W_3)$ such that ${\pi_y}_{\#}\gamma^{12} = \mu_2 = {\pi_x}_{\#}\gamma^{23}$. Then there exists $\gamma^{123} \in \mathcal{P}(W_1 \times W_2 \times W_3)$ such that ${\pi_{xy}}_{\#}\gamma^{123} = \gamma^{12}$ and ${\pi_{yz}}_{\#}\gamma^{123} = \gamma^{23}$. 
\end{lemma}

\section{Oblique Dual Frame Potential}\label{section:dualptential}
In this section, we consider the oblique dual frame potential where $\mathcal{H}$ is the $n$-dimensional complex space $\mathbb{C}^n$ and $W$ and $V$ are subspaces such that $\mathbb{C}^n = W \oplus V^\perp$. 
The frame potential was introduced by Benedetto and Fickus~\cite{benedetto2003finite}, who showed that finite unit-norm tight frames are unique optimizers of the potential. The frame potential is the component-wise $2$-norm of the frame's Gram matrix; generalizations to other $p$-norms and to probabilistic frames appeared in~\cite{ehler2012minimization}.  Such $p$-norms of inner products are also of interest in spherical designs; see, e.g.,~\cite{seidel2001definitions}. Dual frame potentials were introduced in \cite{christensen2020equiangular}, where the authors found a lower bound for the dual $2$-frame potential and the minimizer is just the canonical dual. Their work was followed by \cite{aceska2022cross}, where they call the dual frame potential the ``cross-frame potential'' and further generalized it to fusion frames. In \cite{chen2025probabilistic}, a probabilistic dual $2$-frame potential was introduced. 

Our contribution is that we consider the $p$-potential over oblique dual frames; some of the results we obtain recover the previous results concerning dual frames, since dual frames are oblique dual frames with $W=V$. 
Note that $\mathbf{v}^*$ denotes the conjugate transpose of vector $\mathbf{v} \in \mathbb{C}^n$, $N$ the number of vectors $\{\bv_i\}_i$ (respectively, $\{\bw_i\}_i)$, and $d_W$ the dimension of $W$, where $d_W \le \min\{n, N\}$.
We use $\mathbf{S}^\dagger$ to denote the Moore–Penrose inverse of a matrix $\mathbf{S}$.  By convention,  the inner product $\langle \cdot, \cdot \rangle$ in $\mathbb{C}^n$ is linear in the first input and conjugate linear in the second: for $ {\bf x}, {\bf y} \in\mathbb{C}^n$, $\langle {\bf x}, {\bf y} \rangle = {\bf y}^*{\bf x}$. 
 So the rank-one operator $\bw \bv^\ast$ acts as ${\bf f} \mapsto \langle {\bf f}, \bv \rangle \bw$.

We first start with the definition of the oblique dual $p$-frame potential.  We will show that, for a given frame on $W$, its oblique dual $2$-frame potential is minimized by the canonical oblique dual, which generalizes \cite[Theorem 2.2]{christensen2020equiangular} concerning the dual $2$-frame potential.

\begin{definition}
    Suppose $\mathbb{C}^n = W \oplus V^\perp$ and $p>0$. Let $\{\mathbf{w}_i\}_{i=1}^N$ be a frame for $W$ and $\{\mathbf{v}_i\}_{i=1}^N$ an oblique dual frame of $\{\mathbf{w}_i\}_{i=1}^N$ on $V$. Then the \textit{oblique dual $p$-frame potential} between $\{\mathbf{w}_i\}_{i=1}^N$ and $\{\mathbf{v}_i\}_{i=1}^N$ is defined as
    \begin{equation*}
       \sum_{i=1}^N \sum_{j=1}^N |\langle \mathbf{w}_i, \mathbf{v}_j \rangle |^p. 
    \end{equation*}
\end{definition}
  When $p=2$, we often drop the parameter and refer to the potential as the \textit{oblique dual frame potential}.

The following lemma extends Proposition 24 in \cite{aceska2022cross} about dual frames.

\begin{lemma}\label{lem:diaggram}
     Let $\{\mathbf{w}_i\}_{i=1}^N$ be a frame for $W$ and $\{\mathbf{v}_i\}_{i=1}^N$ an oblique dual frame of $\{\mathbf{w}_i\}_{i=1}^N$ on $V$, where $\mathbb{C}^n = W \oplus V^\perp$. Then
    \begin{equation*}
        \sum_{i=1}^N|\langle \mathbf{w}_i, \mathbf{v}_i \rangle |^2 \geq \frac{d_W^2}{N}.
    \end{equation*}
Furthermore, the equality holds if and only if $\langle \mathbf{w}_i, \mathbf{v}_i \rangle =\frac{d_W}{N}$ for each $i$. 
\end{lemma}
\begin{proof}
 Since $\{\mathbf{v}_i\}_{i=1}^N$ is an oblique dual frame of $\{\mathbf{w}_i\}_{i=1}^N$, then $\bpi_{WV^\perp} = \sum_{i=1}^N \mathbf{w}_i \mathbf{v}_i^*$.
    Therefore, 
     \begin{equation*}
       \sum_{i=1}^N \langle \mathbf{w}_i, \mathbf{v}_i \rangle =  \trace\left(\sum_{i=1}^N \mathbf{w}_i \mathbf{v}_i^*\right) = \trace(\bpi_{WV^\perp}) =d_W.
    \end{equation*}
 By the Cauchy–Schwarz inequality, we have
    \begin{equation*}
       \sum_{i=1}^N|\langle \mathbf{w}_i, \mathbf{v}_i \rangle |^2  \geq \frac{\abs{\sum_{i=1}^N \ip{\bw_i}{\bv_i}}^2}{N} = \frac{d_W^2}{N},
    \end{equation*}
    and the equality holds if and only if the vector $\left(\ip{\bw_i}{\bv_i}\right) _{i=1}^N$ is a scalar multiple of the all-ones vector, that is, there exists a constant $C$ such that $\langle \mathbf{w}_i, \mathbf{v}_i \rangle =C$ for each $i$. Computation gives $C =\frac{d_W}{N}$.
\end{proof}

\begin{proposition}\label{lemma:oblique_dual_potential}
    Let $\{\mathbf{w}_i\}_{i=1}^N$ be a frame for $W$ with frame operator $ \mathbf{S} = \sum_{i=1}^N \mathbf{w}_i\mathbf{w}_i^*$ and $\{\mathbf{v}_i\}_{i=1}^N$ an oblique dual frame of $\{\mathbf{w}_i\}_{i=1}^N$ on $V$, where $\mathbb{C}^n = W \oplus V^\perp$. Then 
    \begin{equation*}
         \sum_{i=1}^N \sum_{j=1}^N |\langle \mathbf{w}_i, \mathbf{v}_j \rangle |^2 \geq d_W,
    \end{equation*}
and the equality holds if and only if $\{\mathbf{v}_j\}_{j=1}^N$ is the canonical oblique dual $\{\bpi_{VW^\perp} \mathbf{S}^{\dagger} \mathbf{w}_j\}_{j=1}^N$. 
\end{proposition}
\begin{proof}
For any $\mathbf{f} \in \mathbb{C}^n$, \cite[Proposition 4]{eldar2003sampling} implies that, among all coefficients $\{c_i\}_{i=1}^N \in \ell^2(\{1, \dots, N\})$ for which $\bpi_{WV^\perp} \mathbf{f} = \sum\limits_{i=1}^N c_i \mathbf{w}_i$,
the coefficient sequence with the minimal $\ell^2$-norm is given by $\{\langle \mathbf{f}, \bpi_{VW^\perp} \mathbf{S}^{\dagger} \mathbf{w}_i \rangle\}_{i=1}^N$. In other words, 
    \begin{equation*}
        \sum_{i=1}^N |c_i|^2 \geq \sum_{i=1}^N |\langle \mathbf{f}, \bpi_{VW^\perp} \mathbf{S}^{\dagger} \mathbf{w}_i \rangle|^2
    \end{equation*}
and the equality holds if and only if $c_i = \langle \mathbf{f}, \bpi_{VW^\perp} \mathbf{S}^{\dagger} \mathbf{w}_i \rangle$ for each $i$. 
Since $\{\mathbf{v}_i\}_{i=1}^N$ is an oblique dual frame of $\{\mathbf{w}_i\}_{i=1}^N$ on $V$, then for any ${\bf f} \in \mathbb{C}^n$, 
$\bpi_{WV^\perp} \mathbf{f} = \sum\limits_{i=1}^N \langle \mathbf{f}, \mathbf{v}_i \rangle \mathbf{w}_i$. Therefore, 
\begin{equation}\label{eqn:minObliqueDual}
    \sum_{j=1}^N |\langle \mathbf{f}, \mathbf{v}_j \rangle|^2 \geq \sum_{j=1}^N |\langle \mathbf{f}, \bpi_{VW^\perp} \mathbf{S}^{\dagger} \mathbf{w}_j \rangle|^2
\end{equation}
and the equality holds if and only if $\mathbf{v}_j = \bpi_{VW^\perp} \mathbf{S}^{\dagger} \mathbf{w}_j$, for each $j$.

Suppose $\mathbf{W} = (\mathbf{w}_1, \mathbf{w}_2, \cdots, \mathbf{w}_N) \in \mathbb{C}^{n \times N}$ and then $\mathbf{S} = \mathbf{W} \mathbf{W}^*$.  Letting $\mathbf{f} = \mathbf{w}_i$ in the above equation and summing over $i$, we have
\begin{equation*}
\begin{split}
    \sum_{i=1}^N \sum_{j=1}^N |\langle \mathbf{w}_i, \mathbf{v}_j \rangle|^2 
    &\geq \sum_{i=1}^N \sum_{j=1}^N |\langle \mathbf{w}_i, \bpi_{VW^\perp} \mathbf{S}^{\dagger} \mathbf{w}_j \rangle|^2 = \sum_{i=1}^N \sum_{j=1}^N |\langle \bpi_{WV^\perp}\mathbf{w}_i, \mathbf{S}^{\dagger} \mathbf{w}_j \rangle|^2\\
    & = \sum_{i=1}^N \sum_{j=1}^N |\langle \mathbf{w}_i, \mathbf{S}^{\dagger} \mathbf{w}_j \rangle|^2 = \|\mathbf{W}^*\mathbf{S}^{\dagger}\mathbf{W}\|_{\text{Frob}}^2 \\
    &=\trace((\mathbf{W}^*\mathbf{S}^{\dagger}\mathbf{W})^*\mathbf{W}^*\mathbf{S}^{\dagger}\mathbf{W}) = \trace(\mathbf{W}^*\mathbf{S}^{\dagger}\mathbf{W}\mathbf{W}^*\mathbf{S}^{\dagger}\mathbf{W})\\
    &= \trace(\mathbf{W}^*\mathbf{S}^{\dagger}\mathbf{S}\mathbf{S}^{\dagger}\mathbf{W}) = \trace(\mathbf{W}\mathbf{W}^*\mathbf{S}^{\dagger}) =\trace(\mathbf{S}\mathbf{S}^{\dagger}) = \trace(\bP_W) \\
    &= d_W,
\end{split}
\end{equation*}
where $\|\cdot\|_{\text{Frob}}$ is the matrix Frobenius norm and the equalities are based on the cyclic property of matrix trace and the facts that $\bpi_{VW^\perp}^* = \bpi_{WV^\perp}$ and $\mathbf{S}\mathbf{S}^{\dagger} = \bP_W$. (See Lemma \ref{lemma:frame_projection} and consider $\mu := \frac{1}{N} \sum_{i=1}^N \delta_{\mathbf{w}_i}$.) 
The equality clearly holds when $\mathbf{v}_j = \bpi_{VW^\perp} \mathbf{S}^{\dagger} \mathbf{w}_j$ for each $j$. Conversely,  if the equality holds, then for each $\mathbf{w}_i$, we must have 
\begin{equation*}
    \sum_{j=1}^N |\langle \mathbf{w}_i, \mathbf{v}_j \rangle|^2 = \sum_{j=1}^N |\langle \mathbf{w}_i, \bpi_{VW^\perp} \mathbf{S}^{\dagger} \mathbf{w}_j \rangle|^2.
\end{equation*}
Then by the equality condition in \cref{eqn:minObliqueDual}, $\mathbf{v}_j = \bpi_{VW^\perp} \mathbf{S}^{\dagger} \mathbf{w}_j$, for each $j$.
\end{proof}

We now extend \cref{lem:diaggram} and \cref{lemma:oblique_dual_potential} from oblique dual $2$-frame potentials to oblique dual $2k$-frame potentials where $k \geq 1$.
\begin{corollary}\label{coro:p_Potential_No_Condition}
Let $\{\mathbf{w}_i\}_{i=1}^N$ be a frame for $W$ with frame operator $ \mathbf{S} = \sum_{i=1}^N \mathbf{w}_i\mathbf{w}_i^*$ and $\{\mathbf{v}_i\}_{i=1}^N$ an oblique dual frame of $\{\mathbf{w}_i\}_{i=1}^N$ on $V$, where $\mathbb{C}^n = W \oplus V^\perp$.
Assume additionally that $p=2k$, where $k \geq 1$. Then 
    \begin{equation*}
        \sum_{i=1}^N  |\langle \mathbf{w}_i, \mathbf{v}_i \rangle |^p \geq  N^{1-p}d_W^p \quad \textrm{and} \quad \sum_{i=1}^N \sum_{j=1}^N |\langle \mathbf{w}_i, \mathbf{v}_j \rangle |^p \geq  N^{2-p}d_W^{\frac{p}{2}}.
    \end{equation*}
Furthermore, when $k > 1$, the left-hand inequality is saturated if and only if  $\ip{\bw_i}{\bv_i}=\frac{d_W}{N}$ for all $i$, and the right-hand inequality is saturated if and only if $|\langle \mathbf{w}_i, \mathbf{v}_j \rangle |$ is constant for all $i$ and $j$ and  $\mathbf{v}_j = \bpi_{VW^\perp} \mathbf{S}^\dagger \mathbf{w}_j$ for each $j$.
\end{corollary}
\begin{proof}
We will just prove the right-hand inequality.  The proof for the left-hand inequality is almost identical but uses \cref{lem:diaggram} to analyze saturation. Since the map $|\cdot|^k: \mathbb{R} \rightarrow \mathbb{R}^+$ is convex,  Jensen's inequality and \cref{lemma:oblique_dual_potential} imply
\begin{align}
    \sum_{i=1}^N \sum_{j=1}^N |\langle \mathbf{w}_i, \mathbf{v}_j \rangle |^p &= N^2\sum_{i=1}^N \sum_{j=1}^N \frac{1}{N^2} |\langle \mathbf{w}_i, \mathbf{v}_j \rangle |^{2k}\nonumber\\
    &\geq  N^{2-2k}\left|\sum_{i=1}^N \sum_{j=1}^N |\langle \mathbf{w}_i, \mathbf{v}_j \rangle |^2 \right|^k\label{eqn:cor34a} \\
    &\geq  N^{2-2k}d_W^{k} = N^{2-p}d_W^{\frac{p}{2}}.\label{eqn:cor34b} 
\end{align}
When $k> 1$, the map $|\cdot|^k: \mathbb{R} \rightarrow \mathbb{R}^+$ is strictly convex. Then saturation of~\eqref{eqn:cor34a}  holds if and only if $|\langle \mathbf{w}_i, \mathbf{v}_j \rangle |$ is constant for any $i$ and $j$, and saturation of~\eqref{eqn:cor34b} follows from \cref{lemma:oblique_dual_potential}. 
\end{proof}

Our next theorem is motivated by \cite[Theorem 2.3 and Proposition 2.4]{christensen2020equiangular}, where the authors obtained a lower bound for the mixed coherence of dual frames. We generalize their results to the oblique dual frame setting and also polish the proof in the standard dual setting.
A family of unit-norm vectors $\{\mathbf{f}_i\}_{i=1}^N$ in $\mathbb{C}^n$ 
is an $(N,n)$-\emph{equiangular tight frame} (ETF) if $\{\mathbf{f}_i\}_{i=1}^N$ is tight and if there exists an $\alpha \ge 0$ such that for any $i \neq j$, $\absip{\bof_i}{\bof_j}^2=\alpha$.
For an arbitrary set of unit-norm vectors $\{\mathbf{f}_i\}_{i=1}^N$ in $\mathbb{C}^n$ where $N \geq n$, the \emph{coherence} is $\max_{i \neq j} |\langle \mathbf{f}_i, \mathbf{f}_j \rangle|^2$.  It has been shown (e.g., \cite{welch1974lower}) that the coherence satisfies
\begin{equation}\label{eq:welch}
    \max_{i \neq j} |\langle \mathbf{f}_i, \mathbf{f}_j \rangle|^2 
    \geq \frac{N-n}{n(N-1)},
\end{equation}
and the equality holds if and only if $\{\mathbf{f}_i\}_{i=1}^N$ is an ETF. The inequality in~\eqref{eq:welch} is known as the \emph{Welch-Rankin bound}.
ETFs have applications in coding theory \cite{massey1991welch}, communication systems \cite{strohmer2003grassmannian}, and quantum information processing \cite{zauner1999grundzuge, renes2004symmetric}. 
The construction and existence of ETFs have gained substantial attention \cite{sustik2007existence, holmes2004optimal,  bodmann2009equiangular, appleby2025constructive}, including the famous Zauner's Conjecture about the existence of an ETF consisting of $n^2$ vectors in $\mathbb{C}^n$ \cite{zauner1999grundzuge, appleby2025constructive}. In the following theorem, we show that the equality in the mixed coherence of oblique duals is saturated if and only if there exists an $(N, d_W)$-ETF and the oblique dual is canonical.

\begin{theorem}\label{thm:mixedCoherence}
     Suppose $W$ and $V$ are subspaces of $\mathbb{C}^n$ such that $\mathbb{C}^n = W \oplus V^\perp$. Let $\{\mathbf{w}_i\}_{i=1}^N$ be a frame for $W$ with frame operator $\mathbf{S}$ and $\{\mathbf{v}_i\}_{i=1}^N$ an oblique dual frame of $\{\mathbf{w}_i\}_{i=1}^N$ on $V$ such that $\langle \mathbf{w}_i, \mathbf{v}_i \rangle = \langle \mathbf{w}_j, \mathbf{v}_j \rangle$ for all $i$ and $j$. Then
    \begin{equation*}
       \underset{i \neq j}{\max} \ |\langle \mathbf{w}_i, \mathbf{v}_j \rangle |^2 \geq \frac{d_W(N-d_W)}{N^2(N-1)}.
    \end{equation*}
Furthermore, the equality holds if and only if any of the following equivalent conditions is true:
\begin{itemize}
    \item[$(1)$] For any $i \neq j$, $|\langle \mathbf{w}_i, \mathbf{v}_j \rangle |$ is constant, and $\mathbf{v}_j = \bpi_{VW^\perp} \mathbf{S}^\dagger \mathbf{w}_j$, for each $j$. 
    \item[$(2)$] The mixed Gram matrix $\mathbf{G} = (\langle \mathbf{w}_i, \mathbf{v}_j \rangle)_{ij}$ between $\{\mathbf{w}_i\}_{i=1}^N$ and $\{\mathbf{v}_i\}_{i=1}^N$ is
    \begin{equation*}
        \mathbf{G} =\frac{d_W}{N} \left(\mathbf{Id}_{N \times N} + \sqrt{\frac{N-d_W}{d_W(N-1)}}\mathbf{Q}\right),
    \end{equation*}
    where $\mathbf{Id}_{N \times N}$ is the identity matrix of size $N \times N$ and $\mathbf{Q}$ is a generalized signature matrix (self-adjoint, a zero diagonal and unimodular entries off the diagonal). 
    \item[$(3)$] For each $i$, define $
\boldsymbol{\psi}_i :=  \sqrt{\tfrac{N}{d_W}} \, (\mathbf{S}^\dagger)^{\frac{1}{2}} \mathbf{w}_i$.
Then $\{\boldsymbol{\psi}_i\}_{i=1}^N$ is an $(N,d_W)$- equiangular tight frame for $W$ and $\mathbf{v}_j = \bpi_{VW^\perp} \mathbf{S}^\dagger \mathbf{w}_j$, for each $j$.
    
\end{itemize}
\end{theorem}
\begin{proof}
Since $\{\mathbf{v}_i\}_{i=1}^N$ is an oblique dual of $\{\mathbf{w}_i\}_{i=1}^N$ on $V$ and the $\langle \mathbf{w}_i, \mathbf{v}_i \rangle$ are constant, we have 
$\langle \mathbf{w}_i, \mathbf{v}_i \rangle = \frac{d_W}{N}$ for each $i$ since
\begin{equation*}
   \sum_{i=1}^N \langle \mathbf{w}_i, \mathbf{v}_i \rangle = \trace\left( \sum_{i=1}^N \mathbf{w}_i \mathbf{v}_i^*\right)=\trace(\bpi_{WV^\perp}) =d_W. 
\end{equation*}
Therefore, 
    \begin{equation*}
    \begin{split}
        \underset{i \neq j}{\text{max}} \ |\langle \mathbf{w}_i, \mathbf{v}_j \rangle |^2 &\geq \frac{1}{N(N-1)} \sum_{i \neq j} |\langle \mathbf{w}_i, \mathbf{v}_j \rangle |^2 \\
        &= \frac{1}{N(N-1)} \left( \sum_{i=1}^N \sum_{j=1}^N | \langle \mathbf{w}_i, \mathbf{v}_j \rangle |^2 - \sum_{i=1}^N |\langle \mathbf{w}_i, \mathbf{v}_i \rangle |^2 \right) \\
        & \geq \frac{1}{N(N-1)} \left( d_W - \frac{d_W^2}{N} \right) = \frac{d_W(N-d_W)}{N^2(N-1)},
    \end{split}
    \end{equation*}
where the last inequality follows from \cref{lemma:oblique_dual_potential} and $d_W \geq \frac{d_W^2}{N}$ (since $d_W \leq N$).
Hence, the equality holds if and only if $|\langle \mathbf{w}_i, \mathbf{v}_j \rangle |$ is constant for all $ i \neq j$ and (by the equality condition in \cref{lemma:oblique_dual_potential}) $\mathbf{v}_j = \bpi_{VW^\perp} \mathbf{S}^\dagger \mathbf{w}_j$ for each $j$, where $\mathbf{S} = \sum\limits_{i=1}^N \mathbf{w}_i\mathbf{w}_i^*$. Therefore, equality is equivalent to $(1)$.

In turn, $(1)$ implies $|\langle \mathbf{w}_i, \mathbf{v}_j \rangle| =  \frac{1}{N}\sqrt{\frac{d_W(N-d_W)}{N-1}}$ for any $i \neq j$, so the mixed Gram matrix $\mathbf{G}$ between $\{\mathbf{w}_i\}_{i=1}^N$ and $\{\mathbf{v}_i\}_{i=1}^N$ can be written as 
    \begin{equation*}
        \mathbf{G} =\frac{d_W}{N} \left(\mathbf{Id}_{N \times N} + \sqrt{\frac{N-d_W}{d_W(N-1)}}\mathbf{Q}\right),
    \end{equation*}
    where $\mathbf{Id}_{N \times N}$ is the identity matrix of size $N \times N$ and $\mathbf{Q}$ is a matrix with zero diagonal and unimodular entries off the diagonal. Further, $(1)$ yields that the $i,j$ entry of the mixed Gram is 
    \begin{align*}
    \lefteqn{\ip{\bw_i}{\bpi_{VW^\perp}\bS^\dagger \bw_j} = \ip{\bpi_{WV^\perp}\bw_i}{\bS^\dagger \bw_j} = \ip{\bw_i}{\bS^\dagger \bw_j}}\\
    &= \ip{(\bS^{1/2})^\dagger\bw_i}{(\bS^{1/2})^\dagger \bw_j} =\overline{\ip{(\bS^{1/2})^\dagger\bw_j}{(\bS^{1/2})^\dagger \bw_i}} =\overline{\ip{\bw_j}{\bpi_{VW^\perp}\bS^\dagger \bw_i}},
    \end{align*}
    so $\bQ$ is self-adjoint and $(1)$ implies  $(2)$. 
Conversely, if $\mathbf{G}$ is given as in $(2)$, then for any $i \neq j$, $ |\langle \mathbf{w}_i, \mathbf{v}_j \rangle| = \frac{d_W}{N}\sqrt{\frac{N-d_W}{d_W(N-1)}}$. Hence
\begin{equation*}
       \underset{i \neq j}{\text{max}} \ |\langle \mathbf{w}_i, \mathbf{v}_j \rangle |^2 = \frac{d_W(N-d_W)}{N^2(N-1)},
    \end{equation*}
    which implies $(1)$. 

Finally, let us show $(2)$ and $(3)$ are equivalent. Define matrices $\mathbf{V}=(\mathbf{v}_1, \cdots, \mathbf{v}_N)$, $\mathbf{W} = (\mathbf{w}_1, \cdots, \mathbf{w}_N)$, and $\boldsymbol{\Psi} = (\boldsymbol{\psi}_1, \cdots, \boldsymbol{\psi}_N)$, where $\boldsymbol{\Psi} = \sqrt{\tfrac{N}{d_W}} \, (\mathbf{S}^\dagger)^{\frac{1}{2}} \mathbf{W}$. Thus, $\boldsymbol{\Psi}^* \boldsymbol{\Psi} = \tfrac{N}{d_W}  \mathbf{W}^* \mathbf{S}^\dagger \mathbf{W}$. If $(2)$ holds, then the equality holds and thus $\mathbf{v}_j = \bpi_{VW^\perp} \mathbf{S}^\dagger \mathbf{w}_j$ for each $j$. Hence  $\mathbf{V} = \bpi_{V W^\perp} \mathbf{S}^{\dagger}\mathbf{W} $ and
\begin{equation*}
    \frac{d_W}{N} \left(\mathbf{I}_{N \times N} + \sqrt{\frac{N-d_W}{d_W(N-1)}}\mathbf{Q}\right) = \mathbf{G} = \mathbf{V}^*\mathbf{W}  = \mathbf{W}^* \mathbf{S}^{\dagger} (\bpi_{W V^\perp}\mathbf{W})  = \mathbf{W}^* \mathbf{S}^{\dagger}\mathbf{W}.
    \end{equation*}
Therefore, the Gram matrix of $\{\boldsymbol{\psi}_i\}_{i=1}^N $ is given by
\begin{equation}\label{eqn:psiETF}
    \boldsymbol{\Psi}^* \boldsymbol{\Psi} = \tfrac{N}{d_W}  \mathbf{W}^* \mathbf{S}^\dagger \mathbf{W} = \mathbf{Id}_{N \times N} + \sqrt{\frac{N-d_W}{d_W(N-1)}}\mathbf{Q}
\end{equation}
and the Welch bound saturation condition for \eqref{eq:welch} implies that $\{\boldsymbol{\psi}_i\}_{i=1}^N$ is an $(N,d_W)$-equiangular tight frame. 
Conversely, if $(3)$ holds, then $\{\boldsymbol{\psi}_i\}_{i=1}^N$ is an $(N,d_W)$-equiangular tight frame and its Gram matrix is given as \cref{eqn:psiETF}. Since $\mathbf{v}_j = \bpi_{VW^\perp} \mathbf{S}^\dagger \mathbf{w}_j$ for each $j$, then  $\mathbf{V} = \bpi_{V W^\perp} \mathbf{S}^{\dagger}\mathbf{W} $ and thus
\begin{equation*}
\begin{split}
     \mathbf{G} = \mathbf{V}^*\mathbf{W} & =  \mathbf{W}^*  \mathbf{S}^{\dagger} (\bpi_{W V^\perp}\mathbf{W}) = \mathbf{W}^* \mathbf{S}^{\dagger}\mathbf{W} \\
     & = \frac{d_W}{N} \boldsymbol{\Psi}^* \boldsymbol{\Psi} =  \frac{d_W}{N} \left(\mathbf{Id}_{N \times N} + \sqrt{\frac{N-d_W}{d_W(N-1)}}\mathbf{Q} \right).
\end{split}
\end{equation*}
\end{proof}

The last corollary is inspired by \cite[Theorem 26]{aceska2022cross} and \cref{lemma:oblique_dual_potential}. Note that the condition of $\langle \mathbf{w}_i, \mathbf{v}_i \rangle$ being constant for each $i$ in \cref{thm:mixedCoherence} and \cref{coro:pPotential} also holds in  \cref{example:paley}: for each $j$,  $\langle \delta_j, \phi_j \rangle = \phi(j-j) = \phi(0)$, where $\delta_j$ is the Dirac distribution at $j$ and $\phi_j(t) = \phi(t-j)$.  In addition, one of the equality conditions in \cref{thm:mixedCoherence} and \cref{coro:pPotential} makes the oblique dual frames $\{\mathbf{w}_i\}_{i=1}^N$ and $\{\mathbf{v}_i\}_{i=1}^N$ equiangular, that is, $|\langle \mathbf{w}_i, \mathbf{v}_j \rangle |$ is constant for any $i \neq j$.

\begin{corollary}\label{coro:pPotential}
Suppose $W$ and $V$ are subspaces of $\mathbb{C}^n$ such that $\mathbb{C}^n = W \oplus V^\perp$. Let $\{\mathbf{w}_i\}_{i=1}^N$ be a frame for $W$ with frame operator $\mathbf{S}$ and $\{\mathbf{v}_i\}_{i=1}^N$ an oblique dual frame of $\{\mathbf{w}_i\}_{i=1}^N$ on $V$ such that $\langle \mathbf{w}_i, \mathbf{v}_i \rangle = \langle \mathbf{w}_j, \mathbf{v}_j \rangle$ for all $i$ and $j$.
Assume additionally that $p=2k$ where $k \geq 1$. Then 
    \begin{equation*}
        \sum_{i=1}^N \sum_{j=1}^N |\langle \mathbf{w}_i, \mathbf{v}_j \rangle |^p \geq \frac{|d_W - \frac{d_W^2}{N} |^{\frac{p}{2}}}{N^{{\frac{p}{2}}-1}(N-1)^{\frac{p}{2}-1}} + \frac{d_W^p}{N^{p-1}}. 
    \end{equation*}
Furthermore, when $k > 1$, the equality holds if and only if for any $i \neq j$, $|\langle \mathbf{w}_i, \mathbf{v}_j \rangle |$ is constant, and $\mathbf{v}_j = \bpi_{VW^\perp} \mathbf{S}^\dagger \mathbf{w}_j$, for each $j$. 
\end{corollary}
\begin{proof}
By the proof of \cref{thm:mixedCoherence}, we know that 
$\langle \mathbf{w}_i, \mathbf{v}_i \rangle = \frac{d_W}{N}$ for each $i$. 
Since the map $|\cdot|^k: \mathbb{R} \rightarrow \mathbb{R}^+$ is convex, Jensen's inequality shows that
    \begin{align}
        \sum_{ i \neq j} |\langle \mathbf{w}_i, \mathbf{v}_j \rangle |^p &= N(N-1) \sum_{   i \neq j} \frac{1}{N(N-1)}|\langle \mathbf{w}_i, \mathbf{v}_j \rangle |^{2k}\nonumber\\
        &\geq  \frac{|\sum_{i \neq j} |\langle \mathbf{w}_i, \mathbf{v}_j \rangle|^2 |^{k}}{N^{k-1}(N-1)^{k-1}} \label{eqn:ppotineq}\\
        &= \frac{|\sum_{i=1}^N \sum_{j=1}^N |\langle \mathbf{w}_i, \mathbf{v}_j \rangle |^2 - \sum_{i=1}^N |\langle \mathbf{w}_i, \mathbf{v}_i \rangle|^2 |^{k}}{N^{k-1}(N-1)^{k-1}} \nonumber\\
         & = \frac{|\sum_{i=1}^N \sum_{j=1}^N |\langle \mathbf{w}_i, \mathbf{v}_j \rangle |^2 - \frac{d_W^2}{N} |^{\frac{p}{2}}}{N^{{\frac{p}{2}}-1}(N-1)^{\frac{p}{2}-1}}.\nonumber
    \end{align}
Thus, 
 \begin{align}
     \sum_{i=1}^N \sum_{j=1}^N |\langle \mathbf{w}_i, \mathbf{v}_j \rangle |^p  &=  \sum_{i \neq j} |\langle \mathbf{w}_i, \mathbf{v}_j \rangle |^p + \sum_{i=1}^N |\langle \mathbf{w}_i, \mathbf{v}_i \rangle|^p \nonumber\\
         &\geq  \frac{|\sum_{i=1}^N \sum_{j=1}^N |\langle \mathbf{w}_i, \mathbf{v}_j \rangle |^2 - \frac{d_W^2}{N} |^{\frac{p}{2}}}{N^{{\frac{p}{2}}-1}(N-1)^{\frac{p}{2}-1}} + \frac{d_W^p}{N^{p-1}} \label{eqn:ppotineq_all_index}\\
         &\geq \frac{|d_W - \frac{d_W^2}{N} |^{\frac{p}{2}}}{N^{{\frac{p}{2}}-1}(N-1)^{\frac{p}{2}-1}} + \frac{d_W^p}{N^{p-1}},\label{eqn:ppotineq_all_index2}
 \end{align}
 where~\eqref{eqn:ppotineq_all_index} follows from~\eqref{eqn:ppotineq} and the fact that $\langle \mathbf{w}_i, \mathbf{v}_i \rangle = \frac{d_W}{N}$ for each $i$, and~\eqref{eqn:ppotineq_all_index2} from \cref{lemma:oblique_dual_potential} and $d_W \geq \frac{d_W^2}{N}$.

When $k > 1$ the map $|\cdot|^k: \mathbb{R} \rightarrow \mathbb{R}^+$ is strictly convex, so equality in~\eqref{eqn:ppotineq} holds if and only if $|\langle \mathbf{w}_i, \mathbf{v}_j \rangle |$ is constant for any $i \neq j$. This then implies the equality condition for~\eqref{eqn:ppotineq_all_index}.  For saturation of~\eqref{eqn:ppotineq_all_index2}, we already have from \cref{thm:mixedCoherence} that $\mathbf{v}_j = \bpi_{VW^\perp} \mathbf{S}^\dagger \mathbf{w}_j$, for each $j$.
\end{proof}

We finish this section by comparing the constants in \cref{coro:p_Potential_No_Condition}, \cref{thm:mixedCoherence}, and \cref{coro:pPotential}. 
In \cref{coro:p_Potential_No_Condition}, if $p>2$ and the right-hand inequality is saturated, then computation shows that for any $ i$ and $j$, $ |\langle \mathbf{w}_i, \mathbf{v}_j \rangle| = \frac{\sqrt{d_W}}{N}$.
By the proof of \cref{thm:mixedCoherence} and \cref{coro:pPotential}, we know that if for each $i$,  $\langle \mathbf{w}_i, \mathbf{v}_i \rangle$ is constant, then $ \langle \mathbf{w}_i, \mathbf{v}_i \rangle =  \frac{d_W}{N}$ for each $i$,
and when the equality holds in \cref{thm:mixedCoherence}, its proof shows that
$$|\langle \mathbf{w}_i, \mathbf{v}_j \rangle | = \frac{d_W}{N}\sqrt{\frac{N-d_W}{d_W(N-1)}}, \ \text{for any} \ i \neq j,$$
and similarly for \cref{coro:pPotential}.

\section{Oblique Dual Probabilistic Frames}\label{section:probabilisticdual}
In this section, we introduce oblique dual probabilistic frames, which generalizes oblique dual frames into the probabilistic frame setting. Throughout the remaining sections, $\mathcal{H}$ is the Euclidean space $\mathbb{R}^n$ and $W$ and $V$ are subspaces such that $\mathbb{R}^n = W \oplus V^\perp$. From now on, the inner product is the standard dot product in $\mathbb{R}^n$. We use $(\cdot)^t$ to denote the transpose of a vector or matrix and $d_W$ the dimension of $W$ where $1 \leq d_W \leq n$.
We first give a characterization motivated by the characterization of oblique duals in \cite[Lemma 3.1]{christensen2004oblique}.
\begin{lemma}\label{lemma:obliqueDualEquiv}
    Suppose $\mathbb{R}^n = W \oplus V^\perp$. Let $\mu \in \mathcal{P}_2(W) $ and $\nu \in \mathcal{P}_2(V)$ be Bessel probability measures with bounds $B_\mu>0$ and $B_\nu>0$, respectively, and $\gamma \in \Gamma(\mu, \nu)$. Then the following are equivalent:
\begin{enumerate}
    \item For any $\mathbf{f} \in W$, $\mathbf{f} = \int_{W \times V} \mathbf{x} \langle \mathbf{y}, \mathbf{f} \rangle d\gamma(\mathbf{x}, \mathbf{y})$. 

    \item For any $\mathbf{f} \in \mathbb{R}^n$, 
   $\bpi_{WV^\perp}\mathbf{f} = \int_{W \times V} \mathbf{x} \langle \mathbf{y}, \mathbf{f} \rangle d\gamma(\mathbf{x}, \mathbf{y})$. 
    
        \item For any $\mathbf{f} \in \mathbb{R}^n$, $\bpi_{VW^\perp}\mathbf{f} = \int_{W \times V} \langle \mathbf{x}, \mathbf{f} \rangle \mathbf{y} d\gamma(\mathbf{x}, \mathbf{y})$. 

    \item For any $\mathbf{f}, \mathbf{g} \in \mathbb{R}^n$, $\langle \bpi_{WV^\perp}\mathbf{f}, \mathbf{g} \rangle = \int_{W \times V} \langle \mathbf{x}, \mathbf{g} \rangle \langle \mathbf{y}, \mathbf{f} \rangle d\gamma(\mathbf{x}, \mathbf{y})$. 

    \item For any $\mathbf{f}, \mathbf{g} \in \mathbb{R}^n$, $\langle \bpi_{VW^\perp}\mathbf{f}, \mathbf{g} \rangle = \int_{W \times V} \langle \mathbf{x}, \mathbf{f} \rangle \langle \mathbf{y}, \mathbf{g} \rangle d\gamma(\mathbf{x}, \mathbf{y})$. 
\end{enumerate}
If any of the equivalent conditions is satisfied, then $\mu$ and $\nu$ are probabilistic frames for $W$ and $V$ with lower bounds $\frac{1}{B_\nu}$ and  $\frac{1}{B_\mu}$, respectively. Furthermore, $\mu$ and ${\bP_W}_\#\nu$ are dual frames for $W$, and ${\bP_V}_\#\mu$ and $\nu$ are dual frames for $V$.
\end{lemma}

\begin{proof}
Note that  $(2)$ implies $(1)$ trivially. To see that $(1)$ implies $(2)$, suppose $(1)$ holds and let $\textbf{f} \in \mathbb{R}^n$. Then $\bpi_{WV^\perp}\mathbf{f} \in W$, so 
    \begin{equation*}
    \begin{split}
        \bpi_{WV^\perp}\mathbf{f} &= \int_{W \times V} \mathbf{x} \langle \mathbf{y}, \bpi_{WV^\perp} \mathbf{f} \rangle d\gamma(\mathbf{x}, \mathbf{y}) = \int_{W \times V} \mathbf{x} \langle \bpi_{WV^\perp}^* \mathbf{y},  \mathbf{f} \rangle d\gamma(\mathbf{x}, \mathbf{y}) \\
        &= \int_{W \times V} \mathbf{x} \langle \bpi_{VW^\perp} \mathbf{y},  \mathbf{f} \rangle d\gamma(\mathbf{x}, \mathbf{y}) = \int_{W \times V} \mathbf{x} \langle \mathbf{y},  \mathbf{f} \rangle d\gamma(\mathbf{x}, \mathbf{y}).
    \end{split}
    \end{equation*}
    
In addition, $(2)$ is equivalent to $ \bpi_{WV^\perp} = \int_{W \times V} \mathbf{x} \mathbf{y}^t d\gamma(\mathbf{x}, \mathbf{y})$. 
Taking the adjoint on both sides leads to 
\begin{equation*}
       \bpi_{VW^\perp} = \bpi_{WV^\perp}^* = \int_{W \times V} \mathbf{y} \mathbf{x}^t d\gamma(\mathbf{x}, \mathbf{y}),
    \end{equation*}
which is equivalent to $(3)$. 

Next, $(2)$ implies $(4)$ trivially. On the other hand, if $(4)$ is true, then for any $\mathbf{f}, \mathbf{g} \in \mathbb{R}^n, $
\begin{equation*}
    \left\langle  \bpi_{WV^\perp}\mathbf{f} - \int_{W \times V} \mathbf{x} \langle \mathbf{y}, \mathbf{f} \rangle d\gamma(\mathbf{x}, \mathbf{y}) , \mathbf{g} \right\rangle = 0, 
\end{equation*}
which implies $(2)$. $(3)$ and $(5)$ are equivalent in a similar way. 

If any of the equivalent conditions is satisfied, then $(4)$ shows that for any $\mathbf{f} \in W$, 
    \begin{equation*}
    \begin{split}
         \| \mathbf{f}\|^4 = \left|\int_{W \times V} \langle \mathbf{x}, \mathbf{f} \rangle \langle \mathbf{y}, \mathbf{f} \rangle d\gamma(\mathbf{x}, \mathbf{y}) \right|^{2}
         &\leq \int_{W} |\langle \mathbf{x}, \mathbf{f} \rangle|^2 d\mu(\mathbf{x}) \int_{V} |\langle \mathbf{y}, \mathbf{f} \rangle|^2 d\nu(\mathbf{y}) \\
     &\leq B_\nu \|\mathbf{f}\|^2 \int_{W} |\langle \mathbf{x}, \mathbf{f} \rangle|^2 d\mu(\mathbf{x}).
    \end{split}
    \end{equation*}
 Therefore, $\mu$ is a probabilistic frame for $W$ with lower bound $\frac{1}{B_\nu}$ and upper bound $M_2(\mu)$. Similarly, $(5)$ shows that $\nu$ is a probabilistic frame for $V$ with lower bound $\frac{1}{B_\mu}$ and upper bound $M_2(\nu)$.
Moreover, $(1)$ tells us that for any $ \mathbf{f} \in W$, 
\begin{equation*}
        \mathbf{f} =  \int_{W \times V} \mathbf{x} \langle \bP_W \mathbf{y}, \mathbf{f} \rangle d\gamma(\mathbf{x}, \mathbf{y}) = \int_{W \times W} \mathbf{x} \langle \mathbf{y}, \mathbf{f} \rangle d\tilde{\gamma}(\mathbf{x}, \mathbf{y}) 
    \end{equation*}
where $\tilde{\gamma}:=(\mathbf{Id}, \bP_W)_\#\gamma \in \Gamma(\mu, {\bP_W}_\#\nu)$.
Therefore, $\mu$ and ${\bP_W}_\#\nu$ are dual frames for $W$. Similarly, one can use $(3)$ to show that  ${\bP_V}_\#\mu$ and $\nu$ are dual frames for $V$ with respect to the coupling $(\bP_V, \mathbf{Id})_\#\gamma \in \Gamma({\bP_V}_\#\mu, \nu)$.
\end{proof}

We give the following definition for oblique dual probabilistic frames. 

\begin{definition}
    Suppose $\mathbb{R}^n = W \oplus V^\perp$ and $\mu \in \mathcal{P}_2(W)$. Then $\nu \in \mathcal{P}_2(V)$ is called an \textit{oblique dual probabilistic frame} of $\mu$ on $V$ if there exists $\gamma \in \Gamma(\mu, \nu)$ such that
    \begin{equation*}
       \bpi_{WV^\perp} = \int_{W \times V} \mathbf{x} \mathbf{y}^t d\gamma(\mathbf{x}, \mathbf{y}).
    \end{equation*}
\end{definition}

Of course, we could have used any of the equivalent conditions from \cref{lemma:obliqueDualEquiv} in this definition.

In what follows, we frequently refer to an oblique dual probabilistic frame as an oblique dual frame or simply an oblique dual. 
In particular, if $T: W \rightarrow V$ is measurable and $T_\#\mu$ is an oblique dual of $\mu$ with respect to $(\mathbf{Id}, T)_\#\mu \in \Gamma(\mu, T_\#\mu)$, then $T_\#\mu$ is called an oblique dual probabilistic frame of \emph{pushforward type}. 

Oblique duals of pushforward type include all oblique duals for finite frames. To see this, suppose $\mu = \frac{1}{N} \sum_{i=1}^N \delta_{\mathbf{w}_i}$ is a probabilistic frame for $W$, $T: W \to V$ is measurable, and 
\[
    \nu = T_\#\mu = \frac{1}{N} \sum\limits_{i=1}^N \delta_{T(\mathbf{w}_i)}
\]
is a probabilistic frame for $V$. If $\nu$ is an oblique dual of $\mu$ with respect to $(\mathbf{Id}, T)_\#\mu$, then 
\begin{equation*}
    \bpi_{WV^\perp} = \int_{W} \mathbf{x} \left(T(\mathbf{x})\right)^t d\mu(\mathbf{x}) =  \sum_{i=1}^N \mathbf{w}_{i} \left(\frac{1}{N}T(\mathbf{w}_i)\right)^t,
\end{equation*}
which implies that $\{\frac{1}{N}T(\mathbf{w}_i)\}_{i=1}^N$ is an oblique dual frame of $\{\mathbf{w}_i\}_{i=1}^N$ on $V$. Under this interpretation, we still need the direct sum structure $\mathbb{R}^n = W \oplus V^\perp$ for the consistent reconstruction of oblique dual frames. Furthermore, for any signal $\mathbf{f} \in \mathbb{R}^n$, if we denote the reconstructed signal $\hat{\mathbf{f}}$ as
\begin{equation*}
    \hat{\mathbf{f}} = \int_{W \times V} \mathbf{x} \langle \mathbf{y, f} \rangle d\gamma(\mathbf{x}, \mathbf{y}) = \frac{1}{N} \sum_{i=1}^N \langle  \mathbf{f}, T(\mathbf{w}_i)  \rangle\mathbf{w}_{i}, 
\end{equation*}
then we must have $\bpi_{WV^\perp} \mathbf{f} = \hat{\mathbf{f}}$, equivalent to the consistent reconstruction by \cref{thm:Oblique_dual_characterization}. 
Similar to \cref{def:consistRecons}, we have the following definition for probabilistic consistent reconstruction, which characterizes oblique dual probabilistic frames.  

\begin{definition}\label{def:probabiConsistRecons}
Let $\mu$ and $\nu$ be probabilistic frames for $W$ and $V$ respectively, where $\mathbb{R}^n = W \oplus V^\perp$. Then $\mu$ and $\nu$ are said to perform \emph{probabilistic consistent reconstruction} if there exists $\gamma \in \Gamma(\mu, \nu)$ such that 
for any $\mathbf{f} \in \mathbb{R}^n$,  $\langle \mathbf{f}, \mathbf{z}  \rangle = \langle \hat{\mathbf{f}},  \mathbf{z} \rangle$  for $\nu$-almost all $\mathbf{z} \in V$, where $\hat{\mathbf{f}}$ is the reconstructed signal given by
    \begin{equation*}
        \hat{\mathbf{f}} = \int_{W \times V} \mathbf{x} \langle \mathbf{y, f} \rangle d\gamma(\mathbf{x}, \mathbf{y}).
    \end{equation*}
\end{definition}

\begin{theorem}
  Let $\mu$ and $\nu$ be probabilistic frames for $W$ and $V$ respectively, where $\mathbb{R}^n = W \oplus V^\perp$. Then $\mu$ and $\nu$ perform probabilistic consistent reconstruction if and only if $\nu$ is an oblique dual probabilistic frame of $\mu$ on $V$. 
\end{theorem}
\begin{proof}
Suppose $\nu$ is an oblique dual probabilistic frame of $\mu$ on $V$. Then there exists $\gamma \in \Gamma(\mu, \nu)$ such that 
\begin{equation*}
    \bpi_{WV^\perp} = \int_{W \times V} \mathbf{x} \mathbf{y}^t d\gamma(\mathbf{x}, \mathbf{y}).
\end{equation*}
Hence, for any $\mathbf{f} \in \mathbb{R}^n$, $\hat{\mathbf{f}} = \bpi_{WV^\perp} \mathbf{f}$, and for all $\mathbf{z} \in V$, 
\begin{equation*}
     \langle \hat{\mathbf{f}}, \mathbf{z}  \rangle = \langle \bpi_{WV^\perp} \mathbf{f}, \mathbf{z}  \rangle = \langle \mathbf{f}, \bpi_{VW^\perp} \mathbf{z} \rangle  = \langle \mathbf{f}, \mathbf{z} \rangle, 
\end{equation*}
which implies that $\mu$ and $\nu$ perform probabilistic consistent reconstruction.

Conversely, if $\mu$ and $\nu$ perform probabilistic consistent reconstruction, then there exists $\tilde{\gamma} \in \Gamma(\mu, \nu)$ such that for any $\mathbf{f} \in \mathbb{R}^n$ and for $\nu$ almost all $\mathbf{z} \in V$, $ \langle \mathbf{f}, \mathbf{z}  \rangle = \langle \hat{\mathbf{f}},  \mathbf{z} \rangle$, where
\begin{equation*}
    \hat{\mathbf{f}} = \int_{W \times V} \mathbf{x} \langle \mathbf{y, f} \rangle d\tilde{\gamma}(\mathbf{x}, \mathbf{y}).
\end{equation*}

For a given $\mathbf{f} \in \mathbb{R}^n$, let 
$\mathscr{A}_{\mathbf{f}}:= \{\mathbf{y} \in V: \langle \hat{\mathbf{f}},  \mathbf{y}  \rangle  \neq \langle \mathbf{f}, \mathbf{y} \rangle \}$. 
By the assumption of probabilistic consistent reconstruction, we have 
$\nu(\mathscr{A}_{\mathbf{f}}) = 0$. Since $\tilde{\gamma} \in \Gamma(\mu, \nu)$, then $\tilde{\gamma}(W \times \mathscr{A}_{\mathbf{f}}) = \nu(\mathscr{A}_{\mathbf{f}}) = 0$. The two functions 
$(\mathbf{x}, \mathbf{y}) \mapsto \mathbf{x} \langle \mathbf{y}, \hat{\mathbf{f}} \rangle$
and 
$(\mathbf{x}, \mathbf{y}) \mapsto \mathbf{x} \langle \mathbf{y}, \mathbf{f} \rangle$
only differ on the $\tilde{\gamma}$-null set $W \times \mathscr{A}_{\mathbf{f}}$, so
\begin{equation} \label{eq:F is a projection}
\int_{W \times V} \mathbf{x} \langle \mathbf{y}, \hat{\mathbf{f}} \rangle 
\, d\tilde{\gamma}(\mathbf{x}, \mathbf{y})
=
\int_{W \times V} \mathbf{x} \langle \mathbf{y}, \mathbf{f} \rangle 
\, d\tilde{\gamma}(\mathbf{x}, \mathbf{y}).
\end{equation}

Now let 
\begin{equation*}
       \mathbf{F}: = \int_{W \times V} \mathbf{x} \mathbf{y}^t d\tilde{\gamma}(\mathbf{x}, \mathbf{y}).
    \end{equation*}
Then, for any $\mathbf{f} \in \mathbb{R}^n$, $\hat{\mathbf{f}} = \mathbf{F} \mathbf{f}$, and
\begin{equation*}
    \mathbf{F}^2 \mathbf{f} = \mathbf{F} \hat{\mathbf{f}} = \int_{W \times V} \mathbf{x} \langle \mathbf{y},  \hat{\mathbf{f}} \rangle d\tilde{\gamma}(\mathbf{x}, \mathbf{y}) = \int_{W \times V} \mathbf{x} \langle \mathbf{y},  \mathbf{f} \rangle d\tilde{\gamma}(\mathbf{x}, \mathbf{y}) = \mathbf{F} \mathbf{f},
\end{equation*}
where the third equality is just \eqref{eq:F is a projection}. Therefore, $\mathbf{F}$ is a projection.

Our goal now is to show $\mathbf{F} = \bpi_{WV^\perp}$, which is equivalent to $\range(\mathbf{F})=W$ and $\kernel(\mathbf{F}) = V^\perp$. 
Since $\tilde{\gamma} \in \Gamma(\mu, \nu)$, $\mu$ is supported on $W$, and $\nu$ is supported on $V$, we know $V^\perp \subseteq \kernel(\mathbf{F})$ and $\range(\mathbf{F}) \subseteq W$.  

We first claim that $\kernel(\mathbf{F})= V^\perp$. To see this, suppose $\mathbf{f}_0 \in \kernel(\mathbf{F})$. If there were $\mathbf{z}_0 \in \supp(\nu)$ so that $\langle \mathbf{f}_0, \mathbf{z}_0 \rangle \neq 0$, continuity of the inner product would imply the existence of an open neighborhood $B_{\mathbf{z}_0}$ of $\mathbf{z}_0$ with $\nu(B_{\mathbf{z}_0})>0$ (since $\mathbf{z}_0 \in \supp(\nu)$) such that for any $\mathbf{z} \in B_{\mathbf{z}_0}$, $ \langle \mathbf{f}_0, \mathbf{z}  \rangle \neq 0 $. On the other hand, probabilistic consistent reconstruction implies that for $\nu$-almost all $\mathbf{z} \in V$, 
\[
    \langle \mathbf{f}_0, \mathbf{z}  \rangle = \langle \hat{\mathbf{f}}_0,  \mathbf{z} \rangle = \langle \mathbf{F}\mathbf{f}_0,  \mathbf{z} \rangle=0.
\]
From this contradiction we conclude that $\langle \mathbf{f}_0, \mathbf{z}_0 \rangle = 0$ for all $\mathbf{z}_0 \in \supp(\nu)$. Since $\spanning(\supp(\nu)) = V$, this implies $\mathbf{f}_0 \in V^\perp$. Then $\kernel(\mathbf{F})= V^\perp$.

By the Rank–Nullity theorem, 
\begin{equation*}
    \Dim(\mathrm{Ran}(\mathbf{F})) + \Dim(V^\perp) = n.
\end{equation*}
Since $\mathbb{R}^n = W \oplus V^\perp$ and $\range(\mathbf{F}) \subset W$, this implies $\range(\mathbf{F}) =W$. Therefore, 
 \begin{equation*}
        \mathbf{F} = \bpi_{WV^\perp} = \int_{W \times V} \mathbf{x} \mathbf{y}^t d\tilde{\gamma}(\mathbf{x}, \mathbf{y}),
    \end{equation*}
which implies that $\nu$ is a probabilistic oblique dual frame of $\mu$ on $V$.
\end{proof}

The following lemma shows that once we have obtained the reconstruction formula for given measures $\mu$ and $\nu$, we can construct an oblique dual frame of $\mu$ on any arbitrary subspace $K$ for which $\mathbb{R}^n = W \oplus K^\perp$. In particular, if $\mu$ is a frame for $W$ and $\nu$ is an oblique dual frame of $\mu$ on $V$, then one can construct an oblique dual of $\mu$ on any subspace $K$ with $\mathbb{R}^n = W \oplus K^\perp$.
\begin{lemma}
   Let $W \subset \mathbb{R}^n$ be a subspace. Suppose $\mu \in \mathcal{P}_2(W)$, $\nu \in \mathcal{P}_2(\mathbb{R}^n)$, and there exists $\gamma \in \Gamma(\mu, \nu)$ such that for any $\mathbf{f} \in W$, 
 \begin{equation*}
        \mathbf{f} = \int_{W \times \mathbb{R}^n} \mathbf{x} \langle \mathbf{y}, \mathbf{f} \rangle d\gamma(\mathbf{x}, \mathbf{y}).
    \end{equation*}
Then for any subspace $K$ for which $\mathbb{R}^n = W \oplus K^\perp$, ${\bpi_{KW^\perp}}_{\#}\nu$ is an oblique dual probabilistic frame of $\mu$ on $K$. 
\end{lemma}

\begin{proof}
 Note that for any $\mathbf{g} \in \mathbb{R}^n$, $ \bpi_{WK^\perp}\mathbf{g} \in W$. Then by the assumption,
\begin{equation*}
        \bpi_{WK^\perp}\mathbf{g} = \int_{W \times \mathbb{R}^n} \mathbf{x} \langle \mathbf{y}, \bpi_{WK^\perp}\mathbf{g} \rangle d\gamma(\mathbf{x}, \mathbf{y}) = \int_{W \times \mathbb{R}^n} \mathbf{x} \langle \bpi_{KW^\perp}\mathbf{y}, \mathbf{g} \rangle d\gamma(\mathbf{x}, \mathbf{y}).
    \end{equation*}
    Now let $\Tilde{\gamma}:= (\mathbf{Id}, \bpi_{KW^\perp})_\#\gamma \in \Gamma(\mu, {\bpi_{KW^\perp}}_{\#}\nu) \subset \mathcal{P}(W \times K)$. Then for any $\mathbf{g} \in \mathbb{R}^n$, 
    \begin{equation*}
        \bpi_{WK^\perp}\mathbf{g} = \int_{W \times K} \mathbf{x} \langle \mathbf{y}, \mathbf{g} \rangle d\Tilde{\gamma}(\mathbf{x}, \mathbf{y}).
    \end{equation*}
    Then by \cref{lemma:obliqueDualEquiv}, ${\bpi_{KW^\perp}}_{\#}\nu$ is an oblique dual probabilistic frame of $\mu$ on $K$. 
\end{proof}

Recall that \cref{Oblique_Characterization} gives a parametrization of all the oblique dual frames for a given frame in $W$. In the following lemma, we give an analogous parametrization of all oblique dual probabilistic frames of pushforward type. Additionally, $(\bpi_{VW^\perp}{\bf S}_\mu^{\dagger})_\#\mu$ is called the \emph{canonical oblique dual probabilistic frame} of $\mu$.

\begin{proposition}\label{lemma:poblique_dual_equiv}
Let $\mathbb{R}^n = W \oplus V^\perp$ and let $\mu \in \mathcal{P}_2(W)$ be a probabilistic frame for $W$ with frame operator ${\bf S}_\mu$. 
 Suppose $T_\# \mu$ is an oblique dual frame of $\mu$ on $V$ where $T: W \rightarrow V$ is measurable.
Then for any $\ {\bf x} \in \mathbb{R}^n$, $T$ precisely satisfies 
   \begin{equation*}
        T({\bf x}) = \bpi_{VW^\perp}{\bf S}_\mu^{\dagger}({\bf x}) + h({\bf x}) - \int_{W} \langle {\bf S}_\mu^{\dagger}{\bf x}, {\bf y} \rangle h({\bf y}) d\mu({\bf y}),
      \end{equation*}
where $h: W \rightarrow V$ is such that $h_\#\mu \in \mathcal{P}_2(V)$.
\end{proposition}
\begin{proof}
    First, note that if $T: W \rightarrow V$ is of the above type, then $T_\# \mu \in \mathcal{P}_2(V)$. Now define $ \gamma = ({\bf Id}, T)_\#\mu \in \Gamma(\mu, T_\#\mu) \subset \mathcal{P}(W \times V)$. 
    Then 
    \begin{equation*}
    \begin{split}
        \int_{W \times V} {\bf x} {\bf y}^t d\gamma({\bf x, y})  = \int_{W} {\bf x} T({\bf x})^t d\mu({\bf x}) &=  {\bf S}_\mu{\bf S}_\mu^{\dagger} \bpi_{WV^\perp}+ \int_{W} {\bf x} h({\bf x})^t d\mu({\bf x}) \\
        &\quad -  \int_{W}   {\bf S}_\mu {\bf S}_\mu^{\dagger} {\bf y} h({\bf y})^t d\mu({\bf y}) = \bpi_{WV^\perp},
    \end{split}
    \end{equation*}
where ${\bf S}_\mu {\bf S}_\mu^{\dagger}$ is the orthogonal projection onto $W$ (see \cref{lemma:frame_projection}). Therefore, $T_\# \mu$ is an oblique dual frame to $\mu$. 

Conversely, if $T_\# \mu$ is an oblique dual to $\mu$, then 
   \begin{equation*}
      \bpi_{VW^\perp} = \bpi_{WV^\perp}^* = \int_{W }  T({\bf y}) {\bf y}^t d\mu({\bf y}).
    \end{equation*}
Note that for any ${\bf x} \in \mathbb{R}^n$, 
    \begin{equation*}
    \begin{split}
         T({\bf x}) 
         &= \bpi_{VW^\perp}{\bf S}_\mu^{\dagger}({\bf x}) + T({\bf x}) -  \bpi_{VW^\perp}{\bf S}_\mu^{\dagger}({\bf x}) \\
         &=\bpi_{VW^\perp}{\bf S}_\mu^{\dagger}({\bf x}) + T({\bf x}) - \int_{W }   \langle {\bf S}_\mu^{\dagger}{\bf x},  {\bf y} \rangle T({\bf y}) d\mu({\bf y}).
    \end{split}
    \end{equation*}
Since $T_\# \mu$ is an oblique dual of $\mu$ on $V$, then $T_\#\mu \in \mathcal{P}_2(V)$. Setting $h=T$ shows that $T: W \rightarrow V$ is of the desired type.
\end{proof}

We present another corollary; the proof is similar to \cref{lemma:poblique_dual_equiv}. 
\begin{corollary}
 Let $\mathbb{R}^n = W \oplus V^\perp$ and let $\mu \in \mathcal{P}_2(W)$ be a probabilistic frame for $W$ with frame operator ${\bf S}_\mu$. 
  Suppose $T_\# \mu$ is an oblique dual frame of $\mu$ on $V$ where $T: W \rightarrow V$ is measurable.
 Then for any $\ {\bf x} \in \mathbb{R}^n$, $T$ precisely satisfies 
   \begin{equation*}
        T({\bf x}) = \bpi_{VW^\perp}{\bf S}_\mu^{\dagger}({\bf x}) + h({\bf x}),
      \end{equation*}
where $h: W \rightarrow V$ is such that $h_\#\mu \in \mathcal{P}_2(V)$ with $ \int_{W } {\bf x} h({\bf x})^t d\mu({\bf x}) = {\bf 0}_{n \times n}$.
\end{corollary}

The following example shows that given a probabilistic frame for $W$, its oblique dual probabilistic frames are not necessarily of pushforward type and can include both discrete and absolutely continuous measures.

\begin{example}\label{ex:non_pushforward}
Let $W = \spanning\left\{\begin{pmatrix}
        1 \\0
    \end{pmatrix}\right\}$ and $V = \spanning\left\{\begin{pmatrix}
        1 \\1
    \end{pmatrix}\right\}$.
    Then $V^\perp = \spanning\left\{\begin{pmatrix}
        1 \\-1
    \end{pmatrix}\right\}$ and $\mathbb{R}^2 = W \oplus V^\perp $. 
The corresponding oblique projection $\bpi_{WV^\perp}$ which maps $\begin{pmatrix}
        1 \\0
    \end{pmatrix}$ to $\begin{pmatrix}
        1 \\0
    \end{pmatrix}$ and $\begin{pmatrix}
        1 \\-1
    \end{pmatrix}$ to $\begin{pmatrix}
        0 \\0
    \end{pmatrix}$
    is given by $
\bpi_{WV^\perp} = 
\begin{pmatrix}
1 & 1 \\
0 & 0
\end{pmatrix}$
and thus $
\bpi_{VW^\perp} = 
\begin{pmatrix}
1 & 0 \\
1 & 0
\end{pmatrix}$. 
Now suppose 
$$\mu = \delta_{(1,0)} \ \text{and} \ \nu = \frac{1}{2}\delta_{(0,0)} + \frac{1}{2}\delta_{(2,2)}.$$
Clearly, $\mu$ is a probabilistic frame for $W$, $\nu$ is a probabilistic frame for $V$, and there does not exist a map $T: W \rightarrow V$ such that $\nu = T_\# \mu$. 
However, $\nu$ is an oblique dual probabilistic frame of $\mu$ with respect to the product measure $\mu \otimes \nu$, because
    \begin{equation*}
         \int_{W} \int_{V} \mathbf{x} \mathbf{y}^t d\mu(\mathbf{x}) d\nu(\mathbf{y}) 
         =  \frac{1}{2}\begin{pmatrix}
        1 \\0
    \end{pmatrix}  \begin{pmatrix}
        2 & 2
    \end{pmatrix} 
    = \begin{pmatrix}
        1 & 1 \\ 0
& 0    \end{pmatrix} =  \bpi_{WV^\perp}.
    \end{equation*}

Furthermore, if $\mu$ is the standard Gaussian measure on $W$, then $\mu$ is a probabilistic frame for $W$ and a sample from $\mu$ is ${\bf x} = \begin{pmatrix}
        X \\0
    \end{pmatrix}$ where $X \sim N(0,1)$. 
Then the frame operator and the associated Moore–Penrose inverse are $
{\bf S}_\mu = 
\int_W {\bf x}{\bf x}^t\, d\mu({\bf x}) = 
\begin{pmatrix}
1 & 0 \\
0 & 0
\end{pmatrix} $ and $
{\bf S}_\mu^{\dagger} = 
\begin{pmatrix}
1 & 0 \\
0 & 0
\end{pmatrix}
$.
Now consider the canonical oblique dual $\nu := (\bpi_{VW^\perp}{\bf S}_\mu^{\dagger})_\# \mu$. 
For ${\bf x} = \begin{pmatrix}
        X \\0
    \end{pmatrix}$, 
$$(\bpi_{VW^\perp}{\bf S}_\mu^{\dagger})({\bf x})
= \bpi_{VW^\perp}\begin{pmatrix}
        X \\0
    \end{pmatrix} =  \begin{pmatrix}
1 & 0 \\
1 & 0
\end{pmatrix} \begin{pmatrix}
        X \\0
    \end{pmatrix} = \begin{pmatrix}
        X \\ X
    \end{pmatrix} \in V. $$
Therefore, $\nu$ is a Gaussian measure supported on $V$ with mean $\begin{pmatrix}
        0 \\0
    \end{pmatrix}$ and covariance matrix  $\begin{pmatrix}
1 & 1 \\
1 & 1
\end{pmatrix}$, 
which is clearly absolutely continuous with respect to the Lebesgue measure for $V$. 
\end{example}

\section{Oblique Dual Probabilistic Frame Potential}\label{section:probabilisticDualPotential}

Recall that given $\mathbf{f} \in \mathcal{H} =W \oplus V^\perp$ and a frame $\{{\bf w}_i\}_{i \in I}$ in $W$, \cite[Proposition~4]{eldar2003sampling} and \cite[Proposition~5.1]{eldar2005general} showed that among all coefficient sequences $\{c_i\}_{i \in I} \in \ell^2(I)$ for which 
$$\bpi_{WV^\perp} \mathbf{f} = \sum_{i \in I} c_i \mathbf{w}_i,$$
the coefficient sequence with the minimal $\ell^2$ energy is $ \{\langle \mathbf{f}, \bpi_{VW^\perp} \mathbf{S}^\dagger \mathbf{w}_i \rangle \}_{i \in I}$, which is induced by the canonical oblique dual frame.  

In the following, we show a similar result for probabilistic frames, with functions $\omega \in L^2(\mu, W)$ substituting for the coefficient sequences.

\begin{proposition}\label{ObliqueDualEnergy}
Let $\mathbb{R}^n = W \oplus V^\perp$ and $\mu \in \mathcal{P}_2(W)$ be a probabilistic frame for $W$ with frame operator ${\bf S}_\mu$. For fixed ${\bf f} \in \mathbb{R}^n$, suppose $\bpi_{WV^\perp}{\bf f} = \int_{W} {\bf x} \omega({\bf x}) d\mu({\bf x})$ for some $\omega \in L^2(\mu, W)$. Then 
   \begin{equation*}
       \int_{W} |\omega({\bf x})|^2 d\mu({\bf x}) =  \int_{W} |\langle {\bf f}, \bpi_{VW^\perp}{\bf S}_\mu^{\dagger}{\bf x} \rangle|^2  d\mu({\bf x}) + \int_{W} \vert \omega({\bf x})-\langle {\bf f}, \bpi_{VW^\perp}{\bf S}_\mu^{\dagger}{\bf x} \rangle \vert^2 d\mu({\bf x}).
    \end{equation*}
   Therefore, 
    \begin{equation*}
       \int_{W} |\omega({\bf x})|^2 d\mu({\bf x}) \geq  \int_{W} |\langle {\bf f}, \bpi_{VW^\perp}{\bf S}_\mu^{\dagger}{\bf x} \rangle|^2  d\mu({\bf x}),
    \end{equation*}
   and equality holds if and only if $\omega ({\bf x}) = \langle {\bf f}, \bpi_{VW^\perp}{\bf S}_\mu^{\dagger}{\bf x} \rangle$ for $\mu$-almost all ${\bf x} \in W$. 
\end{proposition}

\begin{proof}
Since $\mu \in \mathcal{P}_2(W)$, we can interpret $\mu$ as a measure on $\R^n$ and $\mu \in \mathcal{P}_2(\mathbb{R}^n)$. The associated analysis operator $U_\mu: \mathbb{R}^n \rightarrow L^2(\mu, W)$ given by $(U_\mu{\bf x})(\cdot) = \langle {\bf x}, \cdot \rangle \in L^2(\mu, W)$ is well-defined (bounded). Hence, the related (adjoint) synthesis operator $U_\mu^*: L^2(\mu, W) \rightarrow \mathbb{R}^n$ exists and is given by $    U_\mu^*(\psi) = \int_{W} {\bf x}  \psi({\bf x}) d\mu({\bf x})$. Then we have
$\kernel(U_\mu^*) = (\range(U_\mu))^\perp$, where $\range(U_\mu)$ is the range of $U_\mu$ and  $\kernel(U_\mu^*)$ is the kernel of $U_\mu^*$ given by
\begin{equation*}
    \kernel(U_\mu^*) := \left\{\psi \in L^2(\mu, W): U_\mu^*(\psi) = \int_{W} {\bf x}  \psi({\bf x}) d\mu({\bf x}) = {\bf 0}\right\}.
\end{equation*}
For the given $\omega \in L^2(\mu, W)$, 
\begin{equation*}
        \omega(\cdot) = \omega(\cdot) - \langle {\bf f}, \bpi_{VW^\perp}{\bf S}_\mu^{\dagger} (\cdot) \rangle + \langle {\bf f}, \bpi_{VW^\perp}{\bf S}_\mu^{\dagger} (\cdot) \rangle. 
    \end{equation*}
Since 
\begin{equation*}
\begin{split}
     U_\mu^*(\omega(\cdot) - \langle {\bf f}, \bpi_{VW^\perp}{\bf S}_\mu^{\dagger} (\cdot) \rangle) &= \int_{W} {\bf x} \omega({\bf x})  d\mu({\bf x})  - \int_{W} {\bf x} \langle {\bf S}_\mu^{\dagger} \bpi_{WV^\perp}{\bf f},{\bf x} \rangle  d\mu({\bf x}) \\
         &= \bpi_{WV^\perp}{\bf f}-{\bf S}_\mu {\bf S}_\mu^{\dagger}\bpi_{WV^\perp}{\bf f}= \bpi_{WV^\perp}{\bf f}-\bpi_{WV^\perp}{\bf f}=\mathbf{0},
\end{split}
\end{equation*}
where the equality follows from \cref{lemma:frame_projection} that ${\bf S}_\mu {\bf S}_\mu^{\dagger}$ is the orthogonal projection of $\mathbb{R}^n$ onto $W$. 
Therefore, $\omega(\cdot) - \langle {\bf f}, \bpi_{VW^\perp}{\bf S}_\mu^{\dagger}(\cdot) \rangle \in \kernel( U_\mu^*) = (\range(U_\mu))^\perp$.

Since $\langle {\bf f}, \bpi_{VW^\perp}{\bf S}_\mu^{\dagger} (\cdot) \rangle = \langle {\bf S}_\mu^{\dagger} \bpi_{WV^\perp}{\bf f}, \cdot \rangle \in \range(U_\mu)$, it follows that $\omega(\cdot) - \langle {\bf f},\bpi_{VW^\perp}{\bf S}_\mu^{\dagger} (\cdot) \rangle  $ is orthogonal to $\langle {\bf f}, \bpi_{VW^\perp}{\bf S}_\mu^{\dagger}(\cdot) \rangle $ in $L^2(\mu, W)$. Then, by the Pythagorean theorem, 
\begin{equation*}
        \Vert \omega \Vert_{L^2(\mu, W)}^2 = \Vert \omega - \langle {\bf f}, \bpi_{VW^\perp}{\bf S}_\mu^{\dagger}(\cdot) \rangle \Vert_{L^2(\mu, W)}^2 + \Vert \langle {\bf f}, \bpi_{VW^\perp}{\bf S}_\mu^{\dagger}(\cdot) \rangle \Vert_{L^2(\mu, W)}^2.
    \end{equation*}
In other words, 
\begin{equation*}
       \int_{W} |\omega({\bf x})|^2 d\mu({\bf x}) =  \int_{W} |\langle {\bf f}, \bpi_{VW^\perp}{\bf S}_\mu^{\dagger}{\bf x} \rangle|^2  d\mu({\bf x}) + \int_{W} \vert \omega({\bf x})-\langle {\bf f}, \bpi_{VW^\perp}{\bf S}_\mu^{\dagger}{\bf x} \rangle \vert^2 d\mu({\bf x}).
    \end{equation*}
 Therefore, 
    \begin{equation*}
       \int_{W} |\omega({\bf x})|^2 d\mu({\bf x}) \geq  \int_{W} |\langle {\bf f}, \bpi_{VW^\perp}{\bf S}_\mu^{\dagger}{\bf x} \rangle|^2  d\mu({\bf x}) ,
    \end{equation*}
and equality holds if and only if
    \begin{equation*}
       \int_{W} \vert \omega({\bf x})-\langle {\bf f}, \bpi_{VW^\perp}{\bf S}_\mu^{\dagger}{\bf x} \rangle \vert^2 d\mu({\bf x}) = 0,
    \end{equation*}
  which is true if and only if 
$\omega ({\bf x}) = \langle {\bf f}, \bpi_{VW^\perp}{\bf S}_\mu^{\dagger} {\bf x} \rangle$ for $\mu$-almost all ${\bf x} \in W$.
\end{proof}

When $W=V$, \cref{ObliqueDualEnergy} tells us that for a probabilistic frame $\mu$ on $W$ and ${\bf f} \in \mathbb{R}^n$, among all $\omega \in L^2(\mu, W)$ for which $\bP_{W}{\bf f} = \int_{W} {\bf x} \omega({\bf x}) d\mu({\bf x})$,
the function $\langle {\bf f}, {\bf S}_\mu^{\dagger} (\cdot) \rangle$ induced by the canonical dual probabilistic frame has the minimum $L^2(\mu, W)$ energy. This recovers the related result \cite[Proposition~3.4]{chen2025probabilistic}. 

We now give the following definition for the oblique dual probabilistic frame potential. 
\begin{definition}
 Let $\mu \in \mathcal{P}_2(W)$ be a probabilistic frame for $W$ and $\nu \in \mathcal{P}_2(V)$ an oblique dual frame of $\mu$ on $V$. Then the \emph{oblique dual probabilistic frame potential} between $\mu$ and $\nu$ is defined as
    \begin{equation*}
         \int_{W}  \int_{V}  |\langle {\bf x}, {\bf y} \rangle|^2 d\mu({\bf x}) d\nu({\bf y}).
    \end{equation*}
\end{definition}

Note that if $T_\# \mu$ is an oblique dual of $\mu$ on $V$ where $T: W \rightarrow V$ is measurable, then the oblique dual probabilistic frame potential between $\mu$ and $T_\# \mu$ is given by
 \begin{equation*}
        \int_{W}  \int_{W}  |\langle {\bf x}, T({\bf y}) \rangle|^{2} d\mu({\bf x}) d\mu({\bf y}).
    \end{equation*}
    
Given a probabilistic frame $\mu$ for $W$, we are going to prove that its oblique dual probabilistic frame potential is minimized among all oblique duals $T_\# \mu$ of pushforward type if and only if $T({\bf y}) =  \bpi_{VW^\perp}{\bf S}_\mu^{\dagger}{\bf y}$ for $\mu$-almost all ${\bf y} \in W$; that is, when $T_\# \mu$ is the canonical oblique dual $(\bpi_{VW^\perp}{\bf S}_\mu^{\dagger})_\#\mu$ up to a $\mu$-null set. This generalizes similar results in \cref{lemma:oblique_dual_potential} about the oblique dual frame potential, in \cite[Theorem~2.2]{christensen2020equiangular} about the dual frame potential, and in \cite[Theorem~3.5]{chen2025probabilistic} about the probabilistic dual frame potential.

\begin{theorem}\label{DualFramePotential}
Let $W$ and $V$ be subspaces such that $\mathbb{R}^n = W \oplus V^\perp$ where $W$ has dimension $d_W$. Suppose $\mu \in \mathcal{P}_2(W)$ is a probabilistic frame for $W$ and  $T_{\#}\mu$ is an oblique dual frame of $\mu$ on $V$ where $T: W \rightarrow V$ is measurable. Then 
    \begin{equation*}
             \int_{W}  \int_{W}  |\langle {\bf x}, T({\bf y}) \rangle|^{2} d\mu({\bf x}) d\mu({\bf y}) \geq d_W,
    \end{equation*}
and equality holds if and only if $T({\bf y}) =  \bpi_{VW^\perp}{\bf S}_\mu^{\dagger}{\bf y}$ for $\mu$-almost all ${\bf y} \in W$.
\end{theorem}
\begin{proof}
Since $T_{\#}\mu$ is an oblique dual frame to $\mu$, for any fixed ${\bf x} \in \mathbb{R}^n$ we have
    \begin{equation*}
       \bpi_{WV^\perp} {\bf x} = \int_{W} {\bf y} \langle {\bf x}, T({\bf y}) \rangle  d\mu({\bf y}).
    \end{equation*}
    Note that $\langle {\bf x}, T(\cdot) \rangle \in L^2(\mu, W)$, since
    \begin{equation*}
        \int_{W} \vert \langle {\bf x}, T({\bf y}) \rangle \vert^2  d\mu({\bf y}) \leq \Vert {\bf x} \Vert^2    \int_{W} \Vert T({\bf y}) \Vert^2  d\mu({\bf y}) < +\infty. 
    \end{equation*}
    Then, by \cref{ObliqueDualEnergy}, we know that for any fixed ${\bf x} \in \mathbb{R}^n$,
    \begin{equation}\label{DualEnergy2}
         \int_{W}  |\langle {\bf x}, T({\bf y}) \rangle|^2  d\mu({\bf y}) 
          \geq  \int_{W}  |\langle {\bf x}, \bpi_{VW^\perp}{\bf S}_\mu^{\dagger}{\bf y} \rangle|^2  d\mu({\bf y}),
    \end{equation}
and equality holds if and only if $\langle {\bf x}, T({\bf y}) \rangle = \langle {\bf x}, \bpi_{VW^\perp}{\bf S}_\mu^{\dagger}{\bf y} \rangle$ for $\mu$-almost all ${\bf y} \in W$.
Therefore,
 \begin{align}
         \int_{W} \int_{W}  |\langle {\bf x}, T({\bf y}) \rangle|^2   d\mu({\bf y}) d\mu({\bf x}) 
       &\geq \int_{W} \int_{W}  |\langle {\bf x}, \bpi_{VW^\perp}{\bf S}_\mu^{\dagger}{\bf y} \rangle|^2  d\mu({\bf y}) d\mu({\bf x})\label{eqn:lowerbound} \\
       &=  \int_{W} \int_{W}  |\langle {\bf x}, {\bf S}_\mu^{\dagger}{\bf y} \rangle|^2 d\mu({\bf x}) d\mu({\bf y})\nonumber \\
        &=  \int_{W} \int_{W}  |\langle ({\bf S}_\mu^\dagger)^{\frac{1}{2}}{\bf x}, ({\bf S}_\mu^{\dagger})^{\frac{1}{2}}{\bf y} \rangle|^2 d\mu({\bf x}) d\mu({\bf y})\nonumber \\
       &= \int_{W} \Vert ({\bf S}_\mu^\dagger)^{\frac{1}{2}}{\bf y} \Vert^2 d\mu({\bf y}) \label{eqn:lowerbound2}\\
       & = d_W,\label{eqn:lowerbound3}
     \end{align}
where~\eqref{eqn:lowerbound2} is due to the fact that ${({\bf S}_\mu^\dagger)^{\frac{1}{2}}}_\#\mu$ is a Parseval frame for $W=\range(({\bf S}_\mu^\dagger)^{\frac{1}{2}}) $. \cref{lemma:frame_projection} yields ~\eqref{eqn:lowerbound3}, since
\begin{equation*}
\begin{split}
    \int_{W} \Vert ({\bf S}_\mu^\dagger)^{\frac{1}{2}}{\bf y} \Vert^2 d\mu({\bf y}) = \trace \left(({\bf S}_\mu^\dagger)^{\frac{1}{2}} {\bf S}_\mu ({\bf S}_\mu^\dagger)^{\frac{1}{2}}\right) = \trace(\bP_W) = d_W.
\end{split}
\end{equation*}

If $T({\bf y}) =  \bpi_{VW^\perp}{\bf S}_\mu^{\dagger}{\bf y}$ for $\mu$-almost all ${\bf y} \in W$, then the inequality in \eqref{eqn:lowerbound} is an equality. Conversely, if \eqref{eqn:lowerbound} is an equality, then for $\mu$-almost all ${\bf x} \in W$, 
 \begin{equation*}
       \int_{W}  |\langle {\bf x}, T({\bf y}) \rangle|^2  d\mu({\bf y}) = \int_{W}  |\langle {\bf x}, \bpi_{VW^\perp}{\bf S}_\mu^{\dagger}{\bf y} \rangle|^2  d\mu({\bf y}).
\end{equation*}
Hence, for $\mu$-almost every ${\bf x} \in W$ we have equality in  \eqref{DualEnergy2}, and hence
$$\langle {\bf x}, \bP_W(T({\bf y}) - \bpi_{VW^\perp}{\bf S}_\mu^{\dagger}{\bf y}) \rangle =0, \ \text{for $\mu$-almost all ${\bf y} \in W$}. $$

In other words, we have shown that equality in the statement of the theorem implies that $\langle {\bf x}, \bP_W(T({\bf y}) - \bpi_{VW^\perp}{\bf S}_\mu^{\dagger}{\bf y}) \rangle =0$ for $\mu$-almost every ${\bf x}$ and ${\bf y}$ in $W$.
Then for $\mu$-almost all ${\bf y} \in W$, 
$$
A_\mu \|\bP_W(T({\bf y}) - \bpi_{VW^\perp}{\bf S}_\mu^{\dagger}{\bf y}) \|^2 \leq \int_W |\langle {\bf x}, \bP_W(T({\bf y}) - \bpi_{VW^\perp}{\bf S}_\mu^{\dagger}{\bf y}) \rangle|^2 d\mu({\bf x}) = 0 ,
$$
where $A_\mu>0$ is the lower frame bound for $\mu \in \mathcal{P}_2(W)$.
Therefore, for $\mu$-almost all ${\bf y} \in W$, $\bP_W(T({\bf y}) - \bpi_{VW^\perp}{\bf S}_\mu^{\dagger}{\bf y}) = {\bf 0}$. Since $\mathbb{R}^n = W \oplus V^\perp$ is equivalent to $\mathbb{R}^n = V \oplus W^\perp$ (see \cref{lem:directsumangle}), it follows that $T({\bf y}) - \bpi_{VW^\perp}{\bf S}_\mu^{\dagger}{\bf y} \in W^\perp \cap V  = \{\bf 0\}$ for $\mu$-almost all ${\bf y} \in W$. Thus, for $\mu$-almost all ${\bf y} \in W$, $T({\bf y}) = \bpi_{VW^\perp}{\bf S}_\mu^{\dagger}{\bf y}$. 
\end{proof}

The following lemma is a special case of a more general result about approximately dual probabilistic frames~\cite[Lemma 4.2]{chen2025redundancy}. 

\begin{lemma}\label{DualIsFrame}
Let $\mu$ be a probabilistic frame for $W$ with upper frame bound $B>0$.  If $\nu \in \mathcal{P}_2(W)$ is a dual probabilistic frame of $\mu$, then $\nu$ is a probabilistic frame for $W$ with frame bounds $\frac{1}{B}$ and $M_2(\nu)$. 
\end{lemma}

 We are now ready to prove  \cref{dualframepotenial2}, which concerns the oblique dual probabilistic frame potential. For a given probabilistic frame for $W$ with bounds $0<A \leq B$, we show that the oblique dual probabilistic frame potential is bounded below by $d_W \frac{A}{B}$. Since $A \leq B$, this lower bound satisfies
$$
d_W \frac{A}{B} \leq d_W,
$$
where $d_W$ is precisely the lower bound of the oblique dual probabilistic frame potential among all oblique duals of pushforward type in \cref{DualFramePotential}.  This is because there exist oblique dual frames of non-pushforward type (see \cref{ex:non_pushforward}), and the minimization of the oblique dual probabilistic frame potential in \cref{dualframepotenial2} is among a larger set compared to the case in \cref{DualFramePotential}. However,  the effect of oblique dual frames of non-pushforward type becomes negligible when the given probabilistic frame is tight ($A=B$), since these two lower bounds are both $d_W$. 

\begin{theorem}\label{dualframepotenial2}
Let $\mathbb{R}^n = W \oplus V^\perp$. Suppose $\mu \in \mathcal{P}_2(W)$ is a probabilistic frame for $W$ with bounds $0<A \leq B$, and  $\nu \in \mathcal{P}_2(V)$ is an oblique dual of $\mu$ on $V$. Then
\begin{equation*}
    \int_{W}  \int_{V} |\langle {\bf x}, {\bf y} \rangle |^2 d\mu({\bf x}) d\nu({\bf y}) \geq \frac{A}{B}d_W,
\end{equation*}
and equality holds if and only if $\mu$ is a tight probabilistic frame for $W$ with bound $A>0$ and $\nu = ({\bpi_{VW^\perp}{\bf S}^{\dagger}_\mu})_\# \mu$. 
\end{theorem}
\begin{proof}
Note that
\begin{equation*}
    \begin{split}
       \int_{W}  \int_{V} |\langle {\bf x}, {\bf y} \rangle |^2 d\mu({\bf x}) d\nu({\bf y}) = \int_{W}  \int_{V} \trace({\bf x} {\bf x}^t {\bf y} {\bf y}^t)d\mu({\bf x}) d\nu({\bf y}) =\trace({\bf S}_\mu{\bf S}_\nu).
    \end{split}
\end{equation*}
Since $\mu$ is a probabilistic frame for $W$ with bounds $0<A \leq B$, then 
\begin{equation*}
    A \bP_W \preceq {\bf S}_\mu \preceq B \bP_W,
\end{equation*}
where $\bP_W$ is the orthogonal projection of $\mathbb{R}^n$ onto $W$. 
Since ${\bf S}_\mu - A\bP_W$ and ${\bf S}_\nu$ are both positive semi-definite,
\begin{align}
    \trace({\bf S}_\mu {\bf S}_\nu - A \bP_W {\bf S}_\nu) &= \trace(({\bf S}_\mu - A\bP_W){\bf S}_\nu)\nonumber\\ 
    &= \|{\bf S}_\nu^{1/2} ({\bf S}_\mu - A\bP_W)^{1/2}\|^2_{\text{Frob}} \geq 0,\label{eq:trace inequality}
\end{align}
and therefore
\begin{align} 
        \trace({\bf S}_\mu{\bf S}_\nu) &\geq A \trace(\bP_W{\bf S}_\nu) = A \trace(\bP_W{\bf S}_\nu \bP_W)\label{eqn:S_mu_Trace}\\
        &= A \int_{V} \|\bP_W {\bf y}\|^2 d\nu({\bf y}).\nonumber
\end{align}
 Since $\nu$ is an oblique dual frame of $\mu$ on $V$, by \cref{lemma:obliqueDualEquiv} and  \cref{DualIsFrame}, ${\bP_W}_\#\nu$ is a dual probabilistic frame of $\mu$ with respect to some $\gamma \in \Gamma(\mu, {\bP_W}_\#\nu)$ on $W$, with the lower frame bound $\frac{1}{B}$. Now let $\{{\bf e}_i\}_{i=1}^{d_W}$ be an orthonormal basis for $W$. Then
\begin{equation}\label{eqn:ONB_In_W}
   \int_{V} \|\bP_W {\bf y}\|^2 d\nu({\bf y}) 
   =  \sum_{i=1}^{d_W}\int_{W} |\langle {\bf e}_i, {\bf y} \rangle|^2 d {\bP_W}_\#\nu({\bf y}) \geq \sum_{i=1}^{d_W} \frac{1}{B} \Vert{\bf e}_i \Vert^2 =  \frac{1}{B}d_W.
\end{equation}
Combining the above results, we have 
\begin{equation*}
    \begin{split}
       \int_{W}  \int_{V} |\langle {\bf x}, {\bf y} \rangle |^2 d\mu({\bf x}) d\nu({\bf y}) \geq A \int_{V} \|\bP_W {\bf y}\|^2 d\nu({\bf y})  \geq \frac{A}{B}d_W.
    \end{split}
\end{equation*}

For the equality, if $\mu$ is a tight frame for $W$ with frame bound $A$ and $\nu = {\bpi_{VW^\perp}{\bf S}^{\dagger}_\mu}_\# \mu$, then ${\bf S}_\mu = A \bP_W$ and ${\bf S}^{\dagger}_\mu = \frac{1}{A}\bP_W$. Then by \cref{lemma:oblique_proj_identity}, 
$${\bP_W}_\#\nu = (\bP_W \bpi_{VW^\perp}{\bf S}^{\dagger}_\mu)_\# \mu = (\frac{1}{A}\bP_W)_\#\mu.$$
Thus, ${\bP_W}_\#\nu$ is a tight frame for $W$ with bound $\frac{1}{A}$ since $${\bf S}_{{\bP_W}_\#\nu} = \frac{1}{A}\bP_W {\bf S}_\mu \frac{1}{A}\bP_W  = \frac{1}{A}\bP_W.$$ 
Hence, the inequalities in \eqref{eqn:S_mu_Trace} and \eqref{eqn:ONB_In_W} become equalities, and thus
\begin{equation*}
        \int_{W}  \int_{V} |\langle {\bf x}, {\bf y} \rangle |^2 d\mu({\bf x}) d\nu({\bf y}) = A \trace(\bP_W{\bf S}_\nu) = A  \sum_{i=1}^{d_W} \frac{1}{A} \Vert{\bf e}_i \Vert^2 =  d_W.
\end{equation*}

Conversely, if the equality holds, we must have equalities in  \eqref{eq:trace inequality} and \eqref{eqn:ONB_In_W}:
\begin{equation}\label{eqn:equality_Theorem_5_5}
   \|{\bf S}_\nu^{1/2} ({\bf S}_\mu - A\bP_W)^{1/2}\|^2_{\text{Frob}} = 0 \quad \text{and} \quad \int_{V} \|\bP_W {\bf y}\|^2 d\nu({\bf y}) = \frac{1}{B}d_W.
\end{equation}
The first implies that ${\bf S}_\nu^{1/2}({\bf S}_\mu- A\bP_W)^{1/2} = \mathbf{0}_{n \times n}$ and hence
$${\bf S}_\nu({\bf S}_\mu- A\bP_W) = {\bf S}_\nu^{1/2} {\bf S}_\nu^{1/2} ({\bf S}_\mu- A\bP_W)^{1/2}({\bf S}_\mu- A\bP_W)^{1/2}  = \mathbf{0}_{n \times n}.$$

Therefore, $\range({\bf S}_\mu- A\bP_W) \subset \kernel({\bf S}_\nu)$. Since $\nu$ is a probabilistic frame for $V$, \cref{lemma:frame_projection} implies $\kernel({\bf S}_\nu) = V^\perp$ and thus $\range({\bf S}_\mu- A\bP_W) \subset V^\perp$. Since $\range({\bf S}_\mu- A\bP_W) \subset W$ and $W \cap V^\perp = \{\mathbf{0}\}$, we conclude $\range({\bf S}_\mu- A\bP_W) = \{{\mathbf{0}}\}$. So  ${\bf S}_\mu= A\bP_W $, which shows that $\mu$ is a tight frame for $W$ with bound $A>0$. 

Moreover, since ${\bP_W}_\#\nu$ is a dual probabilistic frame of $\mu$ on $W$ with respect to $\gamma \in \Gamma(\mu, {\bP_W}_\#\nu)$, then
\begin{equation*}
    \bP_{W} = \int_{W \times W} {\bf x}{\bf y}^t d\gamma({\bf x, y}).
\end{equation*}
Taking the trace of both sides and using the Cauchy–Schwarz inequality twice, we have
\begin{align}
      d_W = \trace(\bP_{W}) &= \int_{W \times W} \langle {\bf x, y} \rangle d\gamma({\bf x, y}) \nonumber\\
       &\leq \int_{W \times W} \| {\bf x} \| \| {\bf  y} \| d\gamma({\bf x, y}) \label{eqn:csa}\\
      & \leq \sqrt{\int_{W}  \Vert {\bf x}  \Vert ^2  d\mu({\bf x}) \int_{W}  \Vert {\bf y}  \Vert ^2  d{\bP_W}_\#\nu({\bf y})}.\label{eqn:csb}
 \end{align}
Since $\mu$ is a tight frame for $W$ with bound $A=B$, then 
\begin{equation*}
    \int_{W}  \Vert {\bf x}  \Vert ^2  d\mu({\bf x}) = \sum_{i=1}^{d_W} \int_{\mathbb{W}}   |\langle {\bf e}_i, {\bf x} \rangle|^2  d\mu({\bf x}) = \sum_{i=1}^{d_W} A \Vert{\bf e}_i \Vert^2 =  A \ d_W.
\end{equation*}
Therefore, 
\begin{equation*}
   \int_{V}  \Vert {\bP_W\bf y}  \Vert ^2  d\nu({\bf y}) \geq  \frac{1}{A}d_W = \frac{1}{B}d_W. 
\end{equation*}
Combining this with the second equation in \eqref{eqn:equality_Theorem_5_5}, we know that the equalities in the above Cauchy–Schwarz inequalities (i.e., \eqref{eqn:csa} and \eqref{eqn:csb}) must hold. 

Saturation of~\eqref{eqn:csb} implies that the functions $
F({\bf x,y}) := \|{\bf x}\|$ and $ 
G({\bf x,y}) := \|{\bf y}\| $
are linearly dependent in $L^2(\gamma, W \times W)$, i.e.,  there exists a constant $c \geq 0$ such that for $\gamma$-almost all $ ({\bf x,y}) \in W \times W$, $\|{\bf y}\| = c \|{\bf x}\|$. 
Furthermore, saturation of~\eqref{eqn:csa} shows that for $\gamma$-almost all $({\bf x,y}) \in W \times W$, ${\bf y} = c_{\bf x} {\bf x}$,
where $c_{\bf x} \geq 0$ is a constant that may depend on  ${\bf x}$. Therefore, for $\gamma$-almost all $ ({\bf x,y}) \in W \times W$, 
$$\|{\bf y}\| = c \|{\bf x}\| \ \text{and} \  {\bf y} = c_{\bf x}{\bf x}.$$ 
If $\bf x$ is nonzero, then $c_{\bf x}=c$, and if $\bf x =0$, and setting $c_{\bf 0} :=c$ does not affect the validity of the above two equations. Therefore, for $\gamma$-almost all $ ({\bf x,y}) \in W \times W$, ${\bf y} = c {\bf x} = c\bP_W {\bf x},$
which implies $\gamma = (\mathbf{Id}, c\bP_W)_\# \mu \in \Gamma(\mu, {\bP_W}_\# \nu)$. Thus ${\bP_W}_\# \nu = (c\bP_W)_\# \mu$. Since  ${\bP_W}_\#\nu$ is a dual frame of $\mu$ on $W$ with respect to $\gamma = (\mathbf{Id}, c\bP_W)_\# \mu$, 
$$
\bP_W = \int_{W \times W} {\bf x}{\bf y}^t \, d\gamma({\bf x,y})
= \int_{W} {\bf x} (c{\bf x})^t \, d\mu({\bf x})
= c {\bf S}_\mu = c A \bP_W.
$$
Thus, $c = \frac{1}{A}>0$ and ${\bP_W}_\# \nu = \left(\tfrac{1}{A} \bP_W \right)_\# \mu$.
Finally, \cref{lemma:oblique_proj_identity} and the fact that $\nu$ is supported in $V$ give
$$\nu ={\bpi_{VW^\perp}}_\#\nu = ({\bpi_{VW^\perp}\bP_W})_\#\nu = ({\frac{1}{A}\bpi_{VW^\perp}\bP_W})_\#\mu = ({\bpi_{VW^\perp}{\bf S}^{\dagger}_\mu})_\# \mu,$$ 
where the last equality is due to ${\bf S}^\dagger_\mu =\frac{1}{A}\bP_W$.
\end{proof}

When $W=V$, the oblique dual probabilistic frame is just a dual probabilistic frame. Therefore, \cref{dualframepotenial2} tells us that for a given probabilistic frame $\mu$ on $W$, its dual probabilistic frame potential is minimized if and only if $\mu$ is a tight frame and the dual frame is the canonical dual, which recovers the related result of \cite[Theorem 3.7]{chen2025probabilistic}.

\section{Oblique Approximately Dual Probabilistic Frame}\label{section:approximateDual}
We introduce the notion of \emph{oblique approximately dual probabilistic frame} in this section. Previous work on approximate dual frames and approximately dual probabilistic frames can be found in \cite{christensen2010approximately} and \cite{chen2025redundancy}, respectively. Throughout this section, $\mathbb{R}^n = W \oplus V^\perp$ where $W$ and $V$ are subspaces of $\mathbb{R}^n$. 

\begin{definition}
Let $\mu \in \mathcal{P}_2(W)$ and $\nu \in \mathcal{P}_2(V)$ be probabilistic frames for $W$ and $V$, respectively. If $\epsilon \geq 0$, then  $\nu$ is called an \textit{oblique $\epsilon$-approximately dual probabilistic frame} of $\mu$ on $V$ if there exists $\gamma \in \Gamma(\mu, \nu)$ such that
    \begin{equation*}
       \left \| \int_{W \times V} \mathbf{x} \mathbf{y}^t d\gamma(\mathbf{x}, \mathbf{y}) - \bpi_{WV^\perp} \right \| \leq \epsilon,
    \end{equation*}
where the above matrix norm is the operator norm induced by the Euclidean norm. 
\end{definition}
Clearly, when $\epsilon = 0$, the oblique $\epsilon$-approximately dual probabilistic frame is an oblique dual frame.
We can also generalize the probabilistic consistent reconstruction for oblique approximately dual frames.

\begin{definition}
  Let $\mu$ and $\nu$ be probabilistic frames for $W$ and $V$, respectively. Given $\epsilon \geq 0$, $\mu$ and $\nu$ are said to perform \emph{probabilistic $\epsilon$-consistent reconstruction} if there exists $\gamma \in \Gamma(\mu, \nu)$ such that for any $\mathbf{f} \in \mathbb{R}^n$,  
  $$
  \|\langle \mathbf{f} - \hat{\mathbf{f}},\cdot \rangle\|_{L^2(\nu, V)} = \left (\int_{V} |\langle \mathbf{f} - \hat{\mathbf{f}}, \mathbf{z} \rangle|^2 d\nu(\mathbf{z}) \right)^{1/2} \leq \epsilon \|\mathbf{f}\|,
  $$
  where $\hat{\mathbf{f}} = \int_{W \times V} \mathbf{x} \langle \mathbf{y, f} \rangle d\gamma(\mathbf{x}, \mathbf{y}) \in W$ is the reconstructed signal. 
\end{definition}
Note that when $\epsilon =0$, $\|\langle \mathbf{f} - \hat{\mathbf{f}},\cdot \rangle\|_{L^2(\nu, V)}=0$, and thus $\langle \mathbf{f}, \mathbf{z}  \rangle = \langle \hat{\mathbf{f}},  \mathbf{z} \rangle$  for $\nu$-almost all $\mathbf{z} \in V$, which implies that $\mu$ and $\nu$ perform probabilistic consistent reconstruction. We also establish the following lemma on the relation between the oblique $\epsilon$-approximately dual frame and probabilistic $\epsilon$-consistent reconstruction. 
\begin{proposition}
    Let $\mu$ and $\nu$ be probabilistic frames for $W$ and $V$, respectively.
    \begin{itemize}
        \item[$(1)$] If $\nu$ is an oblique $\epsilon$-approximate dual of $\mu$ on $V$, then $\mu$ and $\nu$ do probabilistic $\sqrt{B}\epsilon$-consistent reconstruction where $B>0$ is the upper bound for $\nu$.

        \item[$(2)$] Conversely, if $\mu$ and $\nu$ do probabilistic $\alpha$-consistent reconstruction, then $\nu$ is an oblique $\epsilon$-approximately dual of $\mu$ on $V$, where $\epsilon = \alpha \sqrt{M_2({\bpi_{WV^\perp}{\bf S}_\nu^{\dagger}}_\#\nu)} $. 
    \end{itemize}
\end{proposition}
\begin{proof}
If $\nu$ is an oblique $\epsilon$-approximately dual frame of $\mu$ on $V$, then there exists $\gamma \in \Gamma(\mu, \nu)$ such that 
 \begin{equation*}
      \left \| \int_{W \times V} \mathbf{x} \mathbf{y}^t d\gamma(\mathbf{x}, \mathbf{y}) - \bpi_{WV^\perp} \right \| \leq \epsilon.
    \end{equation*}
For any $\mathbf{f} \in \mathbb{R}^n$, let $\hat{\mathbf{f}} = \int_{W \times V} \mathbf{x} \langle \mathbf{y, f} \rangle d\gamma(\mathbf{x}, \mathbf{y}) \in W$.
Then  $\|\hat{\mathbf{f}} - \bpi_{WV^\perp} \mathbf{f}\| \leq \epsilon \|\mathbf{f}\|$. Therefore, 
    \begin{equation*}
    \begin{split}
        \int_{V} |\langle \mathbf{f} - \hat{\mathbf{f}}, \mathbf{z} \rangle|^2 d\nu(\mathbf{z}) 
        =  \int_{V} |\langle \mathbf{f} - \hat{\mathbf{f}}, \bpi_{VW^\perp}\mathbf{z} \rangle|^2 d\nu(\mathbf{z}) 
       & = \int_{V} |\langle \bpi_{WV^\perp} \mathbf{f} - \hat{\mathbf{f}}, \mathbf{z} \rangle|^2 d\nu(\mathbf{z}) \\
        & \leq B \|\hat{\mathbf{f}} - \bpi_{WV^\perp} \mathbf{f}\|^2 \leq B \epsilon^2 \|\mathbf{f}\|^2,
    \end{split}
    \end{equation*}
which implies that $\mu$ and $\nu$ perform probabilistic $\sqrt{B}\epsilon$-consistent reconstruction. 

Conversely, if $\mu$ and $\nu$ perform probabilistic $\alpha$-consistent reconstruction, then there exists $\gamma' \in \Gamma(\mu, \nu)$ such that for any $\mathbf{f} \in \mathbb{R}^n$, 
$$
  \int_{V} |\langle \mathbf{f} - \hat{\mathbf{f}}, \mathbf{z} \rangle|^2 d\nu(\mathbf{z}) \leq \alpha^2 \|\mathbf{f}\|^2,
  $$
  where $\hat{\mathbf{f}} = \int_{W \times V} \mathbf{x} \langle \mathbf{y, f} \rangle d\gamma'(\mathbf{x}, \mathbf{y})  \in W$. 
Since $ ({\bpi_{WV^\perp}{\bf S}_\nu^{\dagger}})_\#\nu$ is the canonical oblique dual of $\nu$ on $W$ with respect to $\tilde{\gamma} = (\mathbf{Id}, \bpi_{WV^\perp}{\bf S}_\nu^{\dagger})_\#\nu \in \Gamma(\nu, ({\bpi_{WV^\perp}{\bf S}_\nu^{\dagger}})_\#\nu)$, then 
\begin{equation*}
    \int_{V \times W} \mathbf{z} \mathbf{u}^t d\tilde{\gamma}(\mathbf{z}, \mathbf{u}) = \bpi_{VW^\perp} = \bpi_{WV^\perp}^*.
\end{equation*}
Thus, for any $\mathbf{f} \in \mathbb{R}^n$, 
\begin{equation*}
    \begin{split}
        \left \|\int_{W \times V} \mathbf{x} \langle \mathbf{y, f} \rangle d\gamma'(\mathbf{x}, \mathbf{y}) - \bpi_{WV^\perp} \mathbf{f} \right\|^2 &=  \|\bpi_{WV^\perp}(\hat{\mathbf{f}} - \mathbf{f})\|^2 =  \left \| \int_{V \times W} \mathbf{u} \langle \mathbf{z}, \hat{\mathbf{f}} - \mathbf{f} \rangle d\tilde{\gamma}(\mathbf{z}, \mathbf{u}) \right \|^2\\
        &= \left \| \int_{V} \bpi_{WV^\perp}{\bf S}_\nu^{\dagger} \mathbf{z} \ \langle \mathbf{z}, \hat{\mathbf{f}} - \mathbf{f} \rangle d\nu(\mathbf{z}) \right \|^2 \\
        & \leq \int_{V} \|\bpi_{WV^\perp}{\bf S}_\nu^{\dagger} \mathbf{z}\|^2 d \nu(\mathbf{z}) \ \int_{V} |\langle \mathbf{f} - \hat{\mathbf{f}}, \mathbf{z} \rangle|^2 d\nu(\mathbf{z}) \\
        &\leq M_2({\bpi_{WV^\perp}{\bf S}_\nu^{\dagger}}_\#\nu) \ \alpha^2 
        \ \|\mathbf{f}\|^2 = \epsilon^2 \|\mathbf{f}\|^2, 
    \end{split}
\end{equation*}
where $\epsilon = \alpha \sqrt{M_2({\bpi_{WV^\perp}{\bf S}_\nu^{\dagger}}_\#\nu)} $. Then $\nu$ is an oblique $\epsilon$-approximately dual of $\mu$ on $V$ with respect to $\gamma' \in \Gamma(\mu, \nu)$. 
\end{proof}

In the remaining part of this section, we consider the oblique approximately dual frames of perturbed probabilistic frames. We claim that if a probability measure is close to one probabilistic frame in some oblique dual pair, then this probability measure  is an oblique $\epsilon$-approximate dual to the other probabilistic frame in the oblique dual pair where $0<\epsilon <1$, as explained in the following diagram.
\begin{center}
  \begin{tikzcd}[row sep =huge, column sep=huge]
  \nu \in \mathcal{P}_2(V) \arrow[leftrightarrow, r, "\text{oblique dual}"]  & \mu \in \mathcal{P}_2(W) \arrow[leftrightarrow, ld, "\text{oblique $\epsilon$-approximate dual}"] \\
 \eta \in \mathcal{P}_2(V) \arrow[u,leftrightarrow, "\text{close}" ] 
 \end{tikzcd}
\end{center}

In particular, since probabilistic frames can be arbitrarily well approximated by atomic measures with finitely many atoms (e.g., by sampling), and since such measures can be identified with classical frames, this implies that every probabilistic frame has an oblique $\epsilon$-approximate dual which is a classical frame.

Note that the set of probabilistic frames on $V$ is open in $\mathcal{P}_2(V)$ under the $2$-Wasserstein topology, meaning that the frame property is invariant under small perturbations.

\begin{proposition}[{\cite[Proposition 1.2, Corollary 1.3]{chen2023paley}}]\label{prop:quadratic closeness}
    Let $\nu \in \mathcal{P}_2(V)$ be a probabilistic frame for the subspace $V$ with lower  bound $A > 0$. If there exist $\eta \in \mathcal{P}_2(V)$  and $\gamma \in \Gamma(\nu, \eta) $ so that 
$$\lambda:=  \int_{V \times V} \| {\bf y}- {\bf z} \|^2 d\gamma({\bf y},{\bf z})  < A,$$
then $\eta$ is  a probabilistic frame for $V$ with frame bounds $(\sqrt{A}-\sqrt{\lambda})^2 $ and $M_2(\eta)$. Furthermore, if $W_2(\nu, \eta) < \sqrt{A}$,
then $\eta$ is a probabilistic frame for $V$ with bounds $(\sqrt{A}-W_2(\nu, \eta))^2 $ and $M_2(\eta)$.
\end{proposition}

Based on this result, we establish the following lemma.  

\begin{lemma}\label{pertubedFrame}
Suppose $\mu \in \mathcal{P}_2(W)$ is a probabilistic frame for $W$ with upper bound $C>0$, and $\nu  \in \mathcal{P}_2(V)$ is an oblique dual frame of $\mu$ on $V$ with lower bound $A>0$ and $AC\leq 1$.  If $0< \epsilon <1$ and there exist $\eta \in \mathcal{P}_2(V)$ and $\gamma \in \Gamma(\nu, \eta) $ such that
    \begin{equation*}
        \int_{V \times V} \| {\bf y}- {\bf z} \|^2 d\gamma({\bf y},{\bf z})  \leq A \epsilon^2,
    \end{equation*}
 then $\eta$ is an oblique $\epsilon$-approximately dual probabilistic frame of $\mu$ on $V$.
\end{lemma}
\begin{proof}
By \cref{prop:quadratic closeness}, $\eta$ is a probabilistic frame for $V$. Since $\nu$ is an oblique dual probabilistic frame to $\mu$, then there exists  $\gamma' \in \Gamma(\mu, \nu)$ such that
     \begin{equation*}
      \int_{W \times V} {\bf y}{\bf x}^t d\gamma'({\bf x, y}) = \bpi_{VW^\perp} = \bpi_{WV^\perp}^*.
    \end{equation*}
By \cref{gluinglemma}, there exists $\Tilde{\pi} \in \mathcal{P}(W \times V \times V)$ such that ${\pi_{xy}}_{\#}\Tilde{\pi} = \gamma', \  {\pi_{yz}}_{\#}\Tilde{\pi} = \gamma$ where $\pi_{xy}$ and $\pi_{yz}$ are projections to $({\bf x},{\bf y})$ and $({\bf y},{\bf z})$ coordinates. Now let $\Tilde{\gamma}:= {\pi_{xz}}_{\#}\Tilde{\pi} \in \Gamma(\mu, \eta). $
Then, for any ${\bf f} \in \mathbb{R}^n$, 
 \begin{equation*}
 \begin{split}
     \left \| \int_{W \times V} {\bf z} \langle {\bf x}, {\bf f} \rangle d\Tilde{\gamma}({\bf x, z}) - \bpi_{VW^\perp}{\bf f} \right \|^2
      &=  \left \|\int_{W \times V} {\bf z} \langle {\bf x}, {\bf f} \rangle d\Tilde{\gamma}({\bf x, z}) - \int_{W \times V} {\bf y}\langle {\bf x}, {\bf f} \rangle  d\gamma'({\bf x, y}) \right \|^2 \\
      &= \left \|\int_{W \times V \times V} ({\bf z -y}) \langle {\bf x}, {\bf f} \rangle d\Tilde{\pi}({\bf x, y, z})\right \|^2 \\
      & \leq \int_{V \times V}  \| {\bf y -z} \|^2   d\gamma({\bf y, z})   \int_{W}  |\langle {\bf x}, {\bf f} \rangle|^2  d\mu({\bf x}) \\
      & \leq AC \epsilon^2 \| {\bf f}\|^2 \leq \epsilon^2 \| {\bf f}\|^2,
 \end{split}
    \end{equation*}
where the first inequality is due to the triangle inequality and the Cauchy–Schwarz inequality. Then
\begin{equation*}
   \Big \| \int_{W \times V} {\bf x} {\bf z}^t d\Tilde{\gamma}({\bf x, z}) - \bpi_{WV^\perp} \Big \| = \Big \| \int_{W \times V} {\bf z} {\bf x}^t d\Tilde{\gamma}({\bf x, z}) - \bpi_{VW^\perp} \Big \| \leq \epsilon,
\end{equation*}
and hence $\eta$ is an oblique $\epsilon$-approximate dual of $\mu$ on $V$ with respect to $\Tilde{\gamma} \in \Gamma(\mu, \eta)$.
\end{proof}

One may worry about the requirement that the lower frame bound \(A\) for \(\nu\) and the upper frame bound \(C\) for \(\mu\) satisfy 
\(AC \le 1\). However, \cref{lemma:obliqueDualEquiv} shows that if 
\(\mu\) has upper Bessel bound \(C>0\), then \(\nu\) can always be chosen 
with lower frame bound \(A = \frac{1}{C}\). In this case we obtain the 
equality \(AC = 1\). More generally, if \(\nu\) happens to have a smaller 
lower frame bound \(A \le \frac{1}{C}\), then \(AC \le 1\). 

Furthermore, if $\gamma \in \Gamma(\nu, \eta)$ in the above lemma is an optimal transport coupling for $W_2(\nu, \eta)$, we have the following corollary.
\begin{corollary} \label{interiorPoint}
Let $0< \epsilon <1$. Suppose $\mu \in \mathcal{P}_2(W)$ is a probabilistic frame for $W$ with upper bound $C>0$, and $\nu \in \mathcal{P}_2(V)$  is an oblique dual frame to $\mu$ on $V$ with lower bound $A>0$ and $AC\leq 1$.  If there exists $\eta \in \mathcal{P}_2(V)$ such that $W_2(\nu, \eta)  \leq \sqrt{A}\epsilon,$
 then $\eta$ is an oblique $\epsilon$-approximately dual frame of $\mu$ on $V$.
Consequently, for a given probabilistic frame on $W$, its oblique dual probabilistic frames on $V$ are interior points in the set of oblique $\epsilon$-approximately dual frames under the $2$-Wasserstein topology on $\mathcal{P}_2(V)$. 
\end{corollary}

\section*{Acknowledgment}
The authors would like to thank Kasso A. Okoudjou for valuable comments. This work was partially supported by the National Science Foundation (DMS–2107700).
\bibliographystyle{plain} 
\bibliography{refs}
\end{document}